\documentclass[12pt]{amsart}
\usepackage{amsmath,bbm, color}
\usepackage{latexsym,mathrsfs,bm}
\usepackage{hyperref}

\usepackage{url}
\usepackage[english]{babel}
\usepackage{paralist}
\usepackage{amsthm,amsfonts,amssymb,amsmath}
\usepackage{comment}

\usepackage[latin1]{inputenc}

\newcommand{\tRe}{\textup{Re }}
\newcommand{\tIm}{\textup{Im }}

\newcommand{\sumstar}{\sideset{}{^*}\sum}
\newcommand{\sumh}{\sideset{}{^h}\sum}
\newcommand{\sumd}{\sideset{}{^d}\sum}

\newcommand{\es}[1]{\begin{equation}\begin{split}#1\end{split}\end{equation}}
\newcommand{\est}[1]{\begin{equation*}\begin{split}#1\end{split}\end{equation*}}

\renewcommand{\mod}[1]{~\pr{\textnormal{mod}~#1}}
\newtheorem{thm}{Theorem}[section]
\newtheorem{corollary}[thm]{Corollary}
\newtheorem{prop}[thm]{Proposition}
\newtheorem{lem}[thm]{Lemma}
\newtheorem{lemma}[thm]{Lemma}

\theoremstyle{remark}

\newtheorem{rem*}{Remark}
\newcommand{\pr}[1]{\left( #1\right)}

\newcommand{\e}[1]{\operatorname{e}\pr{ #1}}

\newcommand{\bfrac}[2]{\left(\frac{#1}{#2}\right)}

\newcommand{\lt}{\left}
\newcommand{\rt}{\right}

\newcommand{\bsm}{\lt(\begin{smallmatrix}}
\newcommand{\esm}{\end{smallmatrix}\rt)}

\newcommand{\cS}{\mathcal{S}}

\newcommand{\Dis}{\textup{Dis}}
\newcommand{\Ctn}{\textup{Ctn}}
\newcommand{\Hol}{\textup{Hol}}

\def\sumstar{\operatornamewithlimits{\sum\nolimits^*}}
\def\sumb{\operatornamewithlimits{\sum\nolimits^\flat}}
\def\sumh{\operatornamewithlimits{\sum\nolimits^h}}

\let\originalleft\left
\let\originalright\right
\renewcommand{\left}{\mathopen{}\mathclose\bgroup\originalleft}
\renewcommand{\right}{\aftergroup\egroup\originalright}

\numberwithin{equation}{section}

\begin{document}
\title[One-level densities of even and odd orthogonal families]{One-level densities of large even and odd orthogonal families of automorphic $L$-functions}

\date{
\today}

\author[V. Chandee]{Vorrapan Chandee}
\address{Mathematics Department \\ Kansas State University \\ Manhattan, KS 66503}
\email{chandee@ksu.edu}

\author[X. Li]{Xiannan Li}
\address{Mathematics Department \\ Kansas State University \\ Manhattan, KS 66503}
\email{xiannan@ksu.edu }

\author[M.B. Milinovich]{Micah B. Milinovich}
\address{Department of Mathematics \\ University of Mississippi, University, MS 38677}
\email{mbmilino@olemiss.edu }


\allowdisplaybreaks
\numberwithin{equation}{section}
\begin{abstract}
We prove one-level density results for $L$-functions attached to primitive forms of level $q$, averaged over square-free $q$, conditional on the Generalized Riemann Hypothesis (GRH). We treat the even and odd orthogonal families separately and extend the support of the Fourier transform of the test function to $(-3,3)$. This extended support yields the strongest known non-vanishing results for these families of $L$-functions and their derivatives at the central point, conditional on GRH.

\end{abstract}

\dedicatory{\color{black}To Professor Roger Heath-Brown, on the occasion of his 75th birthday}

\maketitle

\section{Introduction} 

In this paper, we are interested in studying the distribution of zeros of $L$-functions.  Here, the guiding heuristic is the Katz-Sarnak  philosophy, which predicts that the zeros of $L$-functions behave like the eigenvalues of random matrices \cite{KaSa}.  Katz and Sarnak verified this for families of $L$-functions over function fields, but the analogous study over number fields is fraught with serious difficulties.

To state our results more precisely, let $S_k(q)$ be the space of cusp forms of fixed weight $k \ge 4$ for the group $\Gamma_0(q)$ with trivial nebentypus,  where as usual,
$$
\Gamma_0(q) := \left\{ \left( \left.{\begin{array}{cc}
   a & b \\
   c & d \\
  \end{array} } \right)  \ \right|  \  ad- bc = 1 , \ \  c \equiv 0 \mod q \right\}.
$$
Let $\mathcal H_k(q) \subset  S_k(q)$ be an orthogonal basis of the space of newforms consisting of Hecke cusp newforms, normalized so that the first Fourier coefficient is $1$. For each $f \in \mathcal H_k(q) $, we let $L(s, f)$ be the $L$-function associated to $f$, defined for Re$(s) > 1$ by
\es{ \label{def:Lsf}
L(s, f) = \sum_{n \geq 1} \frac {\lambda_f(n)}{n^{s}} &= \prod_p \pr{1 - \frac{\lambda_f(p)}{p^s} + \frac{\chi_0(p)}{p^{2s}}}^{-1} \\
&= \prod_p \pr{1 - \frac{\alpha_f(p)}{p^s} }^{-1} \pr{1 - \frac{\beta_f(p)}{p^s} }^{-1},}
where $\{\lambda_f(n)\}$ is the set of Hecke eigenvalues of $f$ and $\chi_0$ denotes the trivial character $\bmod q$. Since $f$ is a newform, $L(s, f)$ can be analytically continued to the entire complex plane and satisfies the functional equation
\es{ \label{eqn:fncL} \Lambda\pr{\tfrac 12 + s, f} = \varepsilon_{f,q} \, \Lambda\pr{ \tfrac 12 - s, f} }
where the completed $L$-function $\Lambda(s, f)$ is defined by
\est{\Lambda\pr{ \tfrac 12 + s, f} = \pr{\frac q{4\pi^2}}^{\frac s2} \Gamma \pr{s + \frac k2} L\pr{ \tfrac 12 + s, f},}
and the sign of the functional equation is $\varepsilon_{f,q}= \pm 1$.  We highlight the dependence on $q$ in $\varepsilon_{f,q}$ since, when $q$ is square-free, we have
\es{\label{def:epsf}
\varepsilon_{f, q} = i^k\mu(q)\lambda_f(q) \sqrt{q}.}
If $\varepsilon_{f,q} = 1,$ then $f$ is called an even form. Otherwise, it is called an odd form. Further, we let $\mathcal H_k^{\pm}(q)$ be a subset of $\mathcal H_k(q)$ of the forms $f$ with $\varepsilon_{f,q} = \pm 1$, respectively. {When $q$ is square-free, the identity \eqref{def:epsf} allows us to restrict sums over $f\in\mathcal H_k(q)$ to sums over $f\in\mathcal H_k^\pm(q)$.}   

Now let $\Phi(x)$ be an even Schwartz class function, and let 
$$
\widehat \Phi(t) = \int_{-\infty}^{\infty} \Phi(x) e^{-2\pi i x t} \> dx
$$
be the usual Fourier transform.  The \textit{Density Conjecture} from the Katz-Sarnak philosophy \cite{KaSa} predicts that

\es{\label{eqn:katzsarnakpredictionevenodd}
\lim_{q\rightarrow \infty} \frac{1}{\#\mathcal{H}_k^{\pm}(q) } \sum_{f\in \mathcal{H}_k^{\pm}(q)} & \sum_{\gamma_f} \Phi \left( \frac{\gamma_f }{2 \pi} \log q\right) \\
&= \int_{-\infty}^{\infty} \Phi(x) W_{G}(x)\,dx \\
&= \int_{-\infty}^{\infty} \Phi(x)\,dx + \frac{1}{2} \int_{-\infty}^{\infty} \widehat \Phi(x) \> dx \mp \frac 12 \int_{|x| > 1}  \widehat \Phi(x) \> dx ,
}
where $W_{even}(x) = 1 + \frac{\sin(2\pi x)}{2\pi x}$ and $W_{odd}(x) = 1 - \frac{\sin(2\pi x)}{2\pi x} + \delta_0(x)$ are the kernels associated with even and odd orthogonal symmetry, respectively. Here $\gamma_f$ runs through the imaginary parts of the nontrivial zeros of $L(s,f)$, each repeated according to multiplicity, and $\delta_0(x)$ is the Dirac delta distribution at $x=0$.  Assuming the Generalized Riemann Hypothesis (GRH) for certain $L$-functions, Iwaniec, Luo, and Sarnak~\cite{ILS} prove that \eqref{eqn:katzsarnakpredictionevenodd} holds when $\widehat \Phi$ is compactly supported in $(-2,2)$ and the variable $q$ runs through the square-free integers.  In \cite{BCL}, the first two authors with Baluyot verified this conjecture for the full orthogonal family, on average over all levels $q$ in an interval, where the support of $\widehat{\Phi}$ is extended to be compactly supported in $(-4, 4)$.  

However, an important motivating application of such results is towards understanding the non-vanishing of $L$-values at the critical point.  The sign of the functional equation implies that $L(1/2, f) = 0$ when $f$ is odd, while conjecturally we expect $L(1/2, f) \neq 0$ for most even forms $f$.  Thus, for this application, it turns out to be more natural and better optimized to study the even and odd families separately.  

To do this, it becomes necessary to include the root number in our average, and for that reason we restrict our attention to square-free level $q$.  In particular, we study a larger family of orthogonal $GL(2)$ $L$-functions by including an average over square-free $q$.  To this end, we fix a smooth function $\Psi$ compactly supported on $\mathbb{R}_{>0}$ with $\widetilde{\Psi}(2) \neq 0$, where
\est{\label{eqn:tildePsi}  \widetilde \Psi(s) = \int_0^{\infty} \Psi(x) x^{s -1} \> dx.}
Let
$$
\mathscr{OL}^{\pm}(Q) := \frac{1}{N^\pm(Q)}\sumb_q \Psi\bfrac{q}{Q} \sum_{f \in \mathcal H_{k}^{\pm}(q)} \sum_{\gamma_f} \Phi \left( \frac{\gamma_f }{2 \pi} \log Q\right),
$$
where $\sumb$ is the sum over square-free $q$, and for $k \geq 4$
\es{\label{def:NpmQ}
N^\pm(Q) := \sumb_q \Psi\bfrac{q}{Q} \sum_{f \in \mathcal H_{k}^\pm(q)} 1.
}

\begin{thm}\label{thm:main} Assume GRH. Let $\Phi$ be an even Schwartz function with $\widehat \Phi$ compactly supported in $(-3, 3).$ Then, with notation as before,
$$ \lim_{Q \rightarrow \infty} \mathscr{OL}^{\pm}(Q) = \int_{-\infty}^{\infty} \Phi(x)\,dx + \frac{1}{2} \int_{-\infty}^{\infty} \widehat \Phi(x) \> dx \mp \frac 12 \int_{|x| > 1}  \widehat \Phi(x) \> dx.
$$
\end{thm}


We can write 
\es{\label{OLsplit}
\mathscr{OL}^{\pm}(Q) := \frac{1}{N^\pm(Q)}\sumb_q \Psi\bfrac{q}{Q} \sumh_{f \in \mathcal H_{k}(q)} \left( \frac{1}{2} \pm \frac{\varepsilon_{f, q}}{2}\right) \sum_{\gamma_f} \Phi \left( \frac{\gamma_f }{2 \pi} \log Q\right).
}
Our main result, Theorem~\ref{thm:main}, follows from the explicit formula and Proposition \ref{prop:main} below. The proof of Theorem~\ref{thm:main} will be given in Section~\ref{sec:proofofMainTheorem}.

{\prop\label{prop:main} 
Assume GRH. Let $\Phi$ be an even Schwartz function with $\widehat \Phi$ compactly supported in $(-3, 3).$ Moreover, let 
\es{\label{def:cfn}
c_f(n) = \left\{ \begin{array}{cc}
     \alpha_f(p)^\ell + \beta_f(p)^\ell  & \ \textrm{if} \  n = p^\ell \\
     0 & \ \textrm{otherwise}, 
\end{array}\right. }
where $\alpha_f$ and $\beta_f$ are as in \eqref{def:Lsf}. Then 
\es{ \label{prop:withoutepsf} \lim_{Q \rightarrow \infty }\frac{1}{N^{\pm}(Q)}  \sumb_q \Psi\bfrac{q}{Q}  \sum_{f \in \mathcal{H}_k(q)} \sum_n \frac{\Lambda(n)c_f(n)}{\sqrt{n}} \widehat \Phi\left( \frac{\log n}{\log Q}\right) = - \frac{1}{2} \int_{-\infty}^{\infty} \widehat \Phi(x) \> dx,  }
and 
\es{ \label{prop:withepsf} \lim_{Q \rightarrow \infty }\frac{1}{N^{\pm}(Q)}  \sumb_q \Psi\bfrac{q}{Q}  \sum_{f \in \mathcal{H}_k(q)} \varepsilon_{f, q} \sum_n \frac{\Lambda(n)c_f(n)}{\sqrt{n}} \widehat \Phi\left( \frac{\log n}{\log Q}\right) = \frac{1}{2} \int_{|x| > 1} \widehat \Phi(x) \> dx.  }
 Here $\Lambda(n)$ is the von Mangoldt function, which is $\log p$ when $n = p^\ell$ and $0$ otherwise.
}

Although the support restriction of $(-3, 3)$ in Theorem \ref{thm:main} goes beyond the results for fixed level in Iwaniec-Luo-Sarnak \cite{ILS}, Barrett et al. \cite{BBDDM} and Cohen et al. \cite{Coh et al}, it is more restrictive than the recent result of the first two authors with Baluyot  \cite{BCL}.  The main obstacle is in understanding the average with the addition of the root number.  It seems to us that this is a natural barrier, arising fundamentally from the fact that we do not see any reduction in conductor in our methods when the support goes beyond $(-3, 3)$.  It is possible that further improvements can be made by using a more powerful Kutznetsov's trace formula on $GL(3)$.  

As mentioned before, one important motivation here is towards non-vanishing, where our results lead to larger proportions of non-vanishing than would be achieved by directly applying the result of Baluyot with the first two authors \cite{BCL}.  The details are in the next section. 

\subsection{Applications to non-vanishing of central values}
It is well-known that one-level density theorems can be used to study the proportion of non-vanishing of $L$-functions and their derivatives in a family at the central point. Using Theorem \ref{thm:main}, we show that at least 69.8\% of the forms in our even orthogonal family forms have $L(\frac{1}{2},f) \ne 0$.  The result with support extended to $(-4, 4)$ from \cite{BCL} only gives a proportion of $(50\!-\! \epsilon)\%$, for any $\epsilon>0,$ following the same calculations as Iwaniec-Luo-Sarnak \cite{ILS}.

By the sign of the functional equation, the forms in our odd orthogonal family satisfy $L(\frac{1}{2},f) = 0$. So for this family, the non-vanishing of the derivative of $L$-function at the central is the more interesting question and we show that at least 97.8\% of the forms in this family have $L'(\frac{1}{2},f) \ne 0$.  

\begin{corollary} \label{non-vanishing} Assume GRH. Then we have
\[
\begin{split}
\liminf_{Q \rightarrow \infty } \frac{1}{N^{+}(Q)}  \sumb_q & \Psi\bfrac{q}{Q} \, \Big| \left\{ f \in \mathcal H_k^{+}(q) : L(\tfrac{1}{2},f) \ne 0 \right\} \Big| 
\\
& \qquad \qquad \ge  \ \frac{19}{24} - \frac{1}{6\!+\!9 \sqrt{2} \tan(\frac{1}{2\sqrt{2}})}  = 0.69818\ldots
\end{split}
\]
and
\[
\begin{split}
\liminf_{Q \rightarrow \infty } \frac{1}{N^{-}(Q)}  \sumb_q & \Psi\bfrac{q}{Q} \, \Big| \left\{ f \in \mathcal H_k^{-}(q) : L'(\tfrac{1}{2},f) \ne 0 \right\} \Big| 
\\
& \qquad \qquad \ge \ \frac{11}{8} - \frac{1}{2\!+\!\sqrt{2} \tan(\frac{1}{2\sqrt{2}})}  = 0.97847\ldots. 
\end{split}
\]
\end{corollary}

In \textsection \ref{non-v}, we indicate how to deduce Corollary \ref{non-vanishing} from Theorem \ref{thm:main}. Using the non-optimal Fourier pair
\[
\Phi(x) = \Big(\frac{\sin 3\pi x}{3 \pi x} \Big)^2 \quad \text{and} \quad \widehat{\Phi}(\xi) = \frac{1}{3} \max\Big(1-\frac{|\xi|}{3},0\Big)
\]
(together with a standard approximation argument) in Theorem \ref{thm:main}, one can obtain the slightly weaker proportions $25/36=0.69444\ldots$ and $35/36=0.97222\ldots$ in place of $0.69818\ldots$ and $0.97847\ldots$, respectively. The better proportions stated in Corollary \ref{non-vanishing} are a result of a solution to a certain Fourier optimization problem that can be solved using methods in \cite{ILS} and \cite{FM2}  (see also \cite{CCM}).


\section{Notation and Preliminary Results} \label{sec:prelimresults}

Throughout the paper, we follow the standard convention in analytic number theory of letting $\epsilon$ denote an arbitrarily small positive real
number whose value may change from line to line. We always use $p$ to denote a prime number. 

In this section, we collect some lemmas that we will need in our proofs. 

\subsection{Orthogonality}
We begin by stating the orthogonality relations for our family. These are the standard Petersson's trace formula (e.g.,~see \cite{Iwaniec}), and a version of Petersson's trace formula that is restricted to newforms, that is due to Iwaniec, Luo, and Sarnak \cite{ILS} for square-free level.

Recall that $S_k(q)$ is the space of cusp forms of weight $k$ and level $q$. Let $B_k(q)$ be any orthogonal basis of $S_k(q)$. Define
\begin{equation*}
\Delta_q(m,n ) = \Delta_{k, q}(m, n) = \sumh_{f\in B_k(q)}\lambda_f(m)\lambda_f(n),
\end{equation*}
where the summation symbol $\sumh$ means that we are summing with the harmonic weight. In particular, for each $f$, we have an associated number $\alpha_f$. We define the harmonic average of $\alpha_f$ over $B_k(q)$ by 
\begin{equation}\label{eqn: harmonicweightsummation}
\sumh_{f \in  B_k(q)} \alpha_f = \frac{\Gamma(k-1)}{(4\pi)^{k-1}}\sum_{f \in B_k(q)} \frac{\alpha_f}{\|f\|^2},
\end{equation}
where  $\|f \|^2 = \int_{\Gamma_0(q) \backslash \mathbb H} |f(z)|^2 y^{k-2} \> dx \> dy. $ The usual version of Petersson's trace formula (e.g.,~see \cite{Iwaniec}) is the following.
\begin{lemma}\label{lem:usualPetersson}
If $m,n,q$ are positive integers, then
$$ \Delta_q(m, n) = \delta(m, n)+ 2\pi i^{-k} \sum_{c\geq 1} \frac{S(m, n;cq)}{cq} J_{k-1}\bfrac{4\pi \sqrt{mn}}{cq},
$$
where $\delta(m,n)=1$ if $m=n$ and is $0$ otherwise, 
\[
S(m,n;q)
=
\sum_{\substack{a \, (\mathrm{mod}\, q) \\ (a,q)=1}}
e\left(\frac{m a + n \bar{a}}{q}\right),
\quad \text{with } a \bar{a} \equiv 1 \textup{ mod }q,
\] 
is the usual Kloosterman sum, and $J_{k-1}$ is the Bessel function of the first kind.
\end{lemma}

Lemma~\ref{lem:usualPetersson}, the Weil bound for Kloosterman sums, and standard facts about the Bessel function imply the following lemma (see Corollary 2.2 in \cite{ILS}).
\begin{lemma}\label{lem:petertruncate}
If $m,n,q$ are positive integers, then
\begin{equation*}
\Delta_q(m, n) = \delta(m, n) + O\left(\frac{\tau(q) (m, n, q)(mn)^{\epsilon}}{q ((m, q)+(n, q))^{1/2}} \bfrac{mn}{\sqrt{mn} + q}^{1/2} \right),
\end{equation*}
where $\tau(q)$ is the divisor function and $\delta(m,n)=1$ if $m=n$ and is $0$ otherwise.
\end{lemma}

For our purposes, we need to isolate the newforms of level $q$ to obtain the sum without the weight.  To be precise, recall that $\mathcal H_k(q)$ is the set of newforms of weight $k$ and level $q$ which are also Hecke eigenforms. We need a formula for
\begin{equation*}
\Delta_q^*( n) := \sum_{f\in \mathcal H_k(q)} \lambda_f(n).
\end{equation*}
We state the result from Proposition 2.11 in \cite{ILS} .
\begin{lemma}\label{lem:PeterssonILS}
Suppose that $n,q$ are positive integers such that $(n,q) | q^2$, and $q$ is square-free. Also, $k$ is a fixed weight. Then

\begin{align*}
\Delta_q^*(n) = \frac{k-1}{12}\sum_{\substack{LM = q }} \frac{\mu(L) M}{ \nu((n, L))}  \sum_{(m, q) = 1 }\frac{\Delta_M(m^2,n)}{m},
\end{align*}
where 
\es{\label{def:nu(j)} \nu(\ell) = \ell \prod_{p | \ell } \left( 1 + \frac 1p\right). }
\end{lemma}

The sum over $m$ in $\Delta_q^*(n)$ is not absolutely convergent, so we need to truncate the sums over $m$ and $L$. This truncation is discussed in Lemma 2.12 of \cite{ILS}, and we restate the result here.

\begin{lem} \label{lem:truncLandm} Assume GRH. Let $k$ be our fixed weight and let $q$ be the square-free level.  Suppose that $n$ is a positive integer such that $(n, q^2)|q $. Moreover, assume that the sequence $\mathscr A = \{a_j\}$ satisfies
\es{\label{propertysequenceaj} 
\sum_{(j, nq) = 1} \lambda_f(j) a_j \ll_{k, \epsilon} (nq)^{\epsilon}
}
for every $f \in \mathcal H_k(M)$ with $M | q$. Define
\es{\label{def:Delta'} \Delta'_q(n) = \frac{k-1}{12} \sum_{\substack{LM = q \\ L \leq X}} \frac{\mu(L) M}{\nu((n, L))} \sum_{ \substack{(m, M) = 1 \\ M \leq Y}} \Delta_{M}(m^2, n). }
Then 
$$ \Delta^*_q(n) = \Delta'_{q}(n) + \Delta^{\infty}_{q}(n)$$
and
$$ \sum_{(j, nq) = 1} \Delta^{\infty}_{q}(nj)a_j   \ll_{k, \epsilon}  (n, q)^{-1/2}  q \, (X^{-1}\!+\! Y^{-1}) (nqXY)^{\epsilon}.$$
\end{lem}

Next are asymptotic formulas for $\Delta_q^*(n)$ and the number of newforms which we take from Proposition 2.13 and Corollary 2.14 in \cite{ILS}.

\begin{lem} \label{lem:sizeofHkq}Suppose $q$ be square-free and $(n, q^2) | q.$ Then

 \es{\label{lem:asymDeltaq*} \Delta_q^*(n) = \delta(n = \square) \frac{k - 1}{12} \frac{\phi(q)}{\sqrt n} + O_k\left( (n, q)^{-1/2} n^{1/6}q^{2/3} \right)}
where $\delta(n = \square)$ is $1$ only if $n$ is a perfect square and vanishes otherwise.  Note that $n$ is a perfect square and $(n, q^2)|q$ with $q$ square-free implies that $(n, q) = 1$.  Further, we have

$$ |\mathcal H_k(q)| = \frac{k-1}{12} \phi(q) + O_k(q^{2/3}),$$
and 
$$ |\mathcal H_k^{\pm}(q)| = \frac{k-1}{24} \phi(q) + O_k(q^{5/6}).$$
    
\end{lem}

Now we derive an asymptotic formula for $N^{\pm}(Q)$ defined in \eqref{def:NpmQ}.

\begin{lem} \label{lem:asympforNQ} Let $\Psi(x)$ be a smooth function compactly supported in $(a, b),$ where $a$ and $b$ are fixed positive constants with $a < b $. Then we have 
\es{ \label{eqn:NpmQintermofNQ} N^{\pm}(Q) = \frac{1}{2}\sumb_q \Psi\left( \frac qQ\right) \sum_{f \in \mathcal H_k(q)} 1 + O_k\left( Q^{11/6}\right)}
and 
\es{ \label{eqn:asympNpmQ} N^{\pm} (Q) = \frac{k-1}{24}Q^2\widetilde \Psi(2) T(2) + O_k\left( Q^{11/6}\right),}
where
\es{\label{def:Ts} T(s) = \prod_{p} \left( 1 - \frac{1}{p^{s}} - \frac{1}{p^{2s - 2}}  + \frac{ 1}{p^{2s - 1}} \right).}
\end{lem}

\begin{proof}
From Lemma \ref{lem:sizeofHkq}, we immediately obtain \eqref{eqn:NpmQintermofNQ}. Further, by Lemma \ref{lem:sizeofHkq}
 
$$ N^{\pm}(Q) =  \frac{k-1}{24} \sumb_q \Psi\left( \frac qQ\right) \phi(q) + O_k(Q^{11/6}). $$
Since $\displaystyle \sum_{a^2 | q} \mu(a)$ is 1 if $q$ is square-free and 0 otherwise, we now rewrite the sum over $q$ as\footnote{Here, and throughout the paper, we write $\int_{(c)}$ to mean $\int_{c-i\infty}^{c+i\infty}$ for $c\in\mathbb{R}$.} 
\est{ \sumb_q \Psi\left( \frac qQ\right) \phi(q)  &= \sum_q \sum_{a^2 | q }  \mu(a) \Psi\left( \frac qQ\right) \phi(q) \\
&= \sum_a \mu(a) \sum_q \Psi\left( \frac {qa^2}Q\right) \phi(qa^2) \\
&= \frac{1}{2\pi i} \int_{(3)} Q^{s}\widetilde{\Psi}(s)  \sum_a \frac{\mu(a)}{a^{2s}} \sum_q \frac{ \phi(qa^2)}{q^s} \> ds\\
&= \frac{1}{2\pi i} \int_{(3)} Q^{s}\widetilde{\Psi}(s)  \zeta(s - 1) T(s) \> ds,}
where $T(s)$ is defined in \eqref{def:Ts}. Note that $T(s)$ is absolutely convergent when $\tRe(s) > 1.$ Moving the contour integration to $\tRe(s) = 1 + \epsilon$ and picking up a residue at $s = 2$, we obtain the lemma.  
\end{proof}

The next lemma collects some well known properties and formulas for the $J$-Bessel function.
{\lem\label{jbessel} Let $J_{k-1}$ be the $J$-Bessel function of order $k-1$. We have
\es{\label{jbessel:bigx} J_{k-1}(2\pi x) =\frac{1}{2\pi \sqrt{x}}\left(W_k(2\pi x) \e{x-\frac{k}{4}+\frac{1}{8}} + \overline{W}_k(2\pi x) \e{-x+\frac{k}{4}-\frac{1}{8} }\right),} where 
$$W_k(x) = \frac{1}{\Gamma(k-1)} \int_0^\infty e^{-u} u^{k - \tfrac 32} \left( 1 + \frac{iu}{4\pi x}\right)^{k - \tfrac 32} \> du$$ 
\textup{(}see \cite[p.~206]{Watt}\textup{)}. Note that $W_k^{(j)}(x)  \ll_{j,k} x^{-j}.$
Moreover, 
\begin{equation}  \label{jbessel:smallx} J_{k-1}(2x) =\sum_{\ell = 0}^{\infty} (-1)^{\ell} \frac{x^{2\ell+k-1}}{\ell! (\ell+k-1)!}
\end{equation}
and $$J_{k-1}(x)\ll \textup{min} \{x^{-1/2}, x^{k-1}\}.$$
Finally,
the Mellin integration representation of $J_{k - 1}$ is 
\begin{equation} \label{eqn:ILMellinforJBessel}
J_{k-1}(x) = \frac{1}{2\pi i} \int_{(\sigma)} 2^{-w - 1} \frac{\Gamma \left( \frac{k - w - 1}{2}\right)}{\Gamma \left( \frac{k +  w + 1}{2} \right)} x^{w} \> dw
\end{equation}
where $0 < \sigma < k - 1$.
}

The proof of the first three claims of Lemma~\ref{jbessel} can be found in \cite[p.~206]{Watt}, and the statement of the last claim is modified from Equation 16 of Table 17.43 in \cite{GR}.  
\vskip 0.1in

\subsection{An explicit formula and some consequences of GRH}

We state the explicit formula, which relates a sum over its prime power coefficient with zeros of $L(s, f)$, and further state some bounds for sums over primes assuming GRH. 

\begin{lem}\label{lem:Explicit} Let $\Phi$ be an even Schwartz function whose Fourier transform has compact support. Suppose the level $q \asymp Q$.  We have
\begin{align*}
    \sum_{\gamma_f} \Phi \left(\frac{\gamma_f}{2\pi} \log Q\right)  &= -\frac{1}{\log Q}\sum_{n=1}^{\infty}\frac{\Lambda(n)[c_f(n) + c_{\bar f}(n)]}{\sqrt{n}} \widehat\Phi\left(\frac{\log n}{\log Q}\right)
    \\
    &\qquad \quad + \int_{-\infty}^{\infty} \Phi(x) \> dx +O_k\left(\frac{1}{\log Q} \right),
\end{align*}
where $c_f(n)$ is defined in \eqref{def:cfn}.
\end{lem}
Lemma 2.5 in \cite{BCL} states a similar result with $\log q$ in place of our $\log Q$.  The proof is very similar, with the observation that $\frac{\log q}{\log Q} = 1 + O\bfrac{1}{\log Q}.$

\begin{lem} \label{lem:CLee3.5} Assume GRH for $L(s,\chi)$ with $\chi$ \textup{mod }$q$ and for $L(s, f)$, where $f$ is a primitive holomorphic Hecke eigenform or a primitive Maass Hecke eigenform of level $q$ and weight $k$. Let $X>0$ be a real number, and let $\Psi$ be a smooth function that is compactly supported on $[0,X]$. Suppose that, for each positive integer $m$, there exists a constant $A_m$ depending only on $m$ such that
$$
|\Psi^{(m)}(x)| \leq \frac{A_m}{\min\{\log(X+3), X/x\}x^{m}}
$$
for all $x>0$. Write $z = \frac 12 +it$ with $t$ real, and let $N$ be a positive integer. If $\chi$ is a non-principal character, then
$$\sum_{(p, N) = 1}\frac{\chi(p)\log(p)\Psi(p)}{p^z}\ll A_3 \log^{1 + \epsilon}( X + 2) \log(q+|t| ) + \log N \max_{0 \leq x \leq X} |\Psi(x)|,$$
with absolute implied constant. Similarly, 
$$\sum_{(p, N) = 1}\frac{\lambda_f(p)\log(p)\Psi(p)}{p^z}\ll A_3 \log^{1 + \epsilon}( X + 2) \log (q+k+|t| ) + \log N \max_{0 \leq x \leq X} |\Psi(x)|,$$
with absolute implied constant.
\end{lem}

\noindent \textit{Remark}: If $c$ is a fixed constant and $\Upsilon$ is a smooth function compactly supported on $[0,c]$, then the function $\Psi(x)=\Upsilon(cx/X)$ satisfies the conditions in Lemma~\ref{lem:CLee3.5} since $X^{-m} \ll x^{-m} (x/X)$ for positive integers $m$. Also, if $\Upsilon$ is a smooth function compactly supported on $(-\infty,c]$, then the function $\Psi(x)=\Upsilon(\frac{c\log x}{\log X})$ satisfies the conditions in the lemma.

\section{Kuznetsov's formula} \label{sec:kutz}

In this section, we will state some relevant results from spectral theory. We refer the reader to \cite{DI} and \cite{Iwaniec} for background reading.  We first introduce some notation that will appear in Kuznetsov's trace formula. There are three parts in Kuznetsov's trace formula -- contributions from holomorphic forms, Maass forms, and Eisenstein series -- and we now define the Fourier coefficients of these forms.

\subsubsection*{Holomorphic forms}

Let $B_{\ell}(q)$ be an orthonormal basis of the space of holomorphic cusp forms of weight $\ell$ level $q$, and $\theta_{\ell}(q)$ is the dimension of the space $S_{\ell}(q)$. We write $B_{\ell}(q) = \{f_1, f_2,...., f_{\theta_{\ell}(q)} \}$, and the Fourier expansion of $f_j \in B_{\ell}(q)$ is written as
$$ f_j(z) = \sum_{n \geq 1} \psi_{j, \ell}(n) (4\pi n)^{\ell / 2} \e{n z},$$
where $e(x)=e^{2\pi i x}$. We call $f$ a Hecke eigenform if it is an eigenfunction of all the Hecke operators $T(n)$ for $(n, q) = 1$.

\subsubsection*{Maass forms}
Let 
$
\lambda_j := \frac{1}{4}+\kappa_j^2,
$
where
$
0=\lambda_0 \leq \lambda_1\leq \lambda_2 \leq \dots
$
are the eigenvalues, each repeated according to multiplicity, of the Laplacian $-y^2 ( \frac{\partial^2}{\partial x^2} + \frac{\partial^2}{\partial y^2})$ acting as a linear operator on the space of cusp forms in $L^2(\Gamma_0(q) \backslash \mathbb{H})$, where by convention we choose the sign of $\kappa_j$ that makes $\kappa_j\geq 0$ if $\lambda_j\geq \frac{1}{4}$ and $i\kappa_j >0$ if $\lambda_j <\frac{1}{4}$. For each of the positive $\lambda_j$, we may choose an eigenvector $u_j$ in such a way that the set $\{u_1,u_2,\dots\}$ forms an orthonormal system, and we define $\rho_j(m)$ to be the $m$th Fourier coefficient of $u_j$, i.e.,
$$
u_j(z) =  \sum_{m\neq 0} \rho_j(m)W_{0, i\kappa_j} (4\pi |m|y) \e{mx}
$$
with $z=x+iy$, where $W_{0, it}(y) =  \left( y/\pi\right)^{1/2}K_{it}(y/2)$ is a Whittaker function, and $K_{it}$ is the modified Bessel function of the second kind. 

We call $u$ a Hecke eigenform if it is an eigenfunction of all the Hecke operators $T(n)$ for $(n, q) = 1$.

\subsubsection*{Eisenstein series}
Let $\mathfrak{c}$ be a cusp for $\Gamma_0(q)$.  We define $\varphi_{\mathfrak{c}}(m, t)$ to be the $m$th Fourier coefficient of the (real-analytic) Eisenstein series at $1/2 + it$:
\begin{align*}
E_{\mathfrak{c}}(z; 1/2 + it)  &=
  \delta_{\mathfrak{c} = \infty} y^{1/2 + it} + \varphi_{\mathfrak{c}}(0,t) y^{1/2 - it}  + \sum_{m\neq 0}  \varphi_{\mathfrak{c}} (m, t) W_{0, it} (4\pi |m|y) \e{mx},
\end{align*}
where $z=x+iy$.

\subsubsection*{Kuznetsov's trace formula} We state the version given by Lemma~10 of \cite{BM}.
\begin{lem}\label{lem:kuznetsov}
Let $\phi:(0,\infty)\rightarrow \mathbb{C}$ be smooth and compactly supported, and let $m,n,q$ be positive integers. Then
\begin{align*}
\sum_{\substack{c\geq 1 \\ c\equiv 0 \bmod{q}}} \frac{S(m,n;c)}{c} \phi
& \bigg( 4\pi \frac{\sqrt{mn}}{c}\bigg) = \sum_{j=1}^{\infty} \frac{\overline{\rho_j}(m)\rho_j(n) \sqrt{mn}}{\cosh(\pi \kappa_j)}\phi_+(\kappa_j)\\
& + \frac{1}{4\pi} \sum_{\mathfrak{c}} \int_{-\infty}^{\infty} \frac{ \sqrt{mn}}{\cosh (\pi t)} \, \overline{\varphi_{\mathfrak{c}} (m, t)}  \varphi_{\mathfrak{c} }(n, t)  \phi_+ (t) \,dt \\
& + \sum_{ \substack{\ell \geq 2 \mbox{\scriptsize{ \upshape{even}}} \\ 1 \leq j \leq \theta_{\ell}(q)} } (\ell-1)! \sqrt{mn} \, \overline{\psi_{j,\ell}}(m) \psi_{j,\ell} (n) \phi_h(\ell),
\end{align*}
where the Bessel transforms $\phi_+$ and $\phi_h$ are defined by
$$
\phi_+(r):=\frac{2\pi i}{\sinh(\pi r)} \int_0^{\infty} (J_{2ir}(\xi) - J_{-2ir} (\xi) ) \phi(\xi) \,\frac{d\xi}{\xi}
$$
and
$$
\phi_h(\ell) := 4 i^k \int_0^{\infty} J_{\ell - 1}(\xi) \phi(\xi) \,\frac{d\xi}{\xi},
$$
where $J_{\ell - 1}$ is the Bessel function of the first kind. 
\end{lem}

\section{Initial Setup for the Proof of Proposition \ref{prop:main}} \label{sec:mainsetup}

By the same arguments as Lemma 4.1 of \cite{BCL} and from Equation (4.23) in \cite{ILS} \footnote{Equation (4.23) contains a minor typo: it should be $p \nmid N$ instead of $p |N$,} we have the following.

\begin{lem}\label{lem:trimpsum}
Let the notation be as in Lemma~\ref{lem:Explicit}, and $q \sim Q$, where $Q \geq 20$. Assuming GRH, we have
\begin{align*}
\frac{1}{\log Q}\sum_n \frac{\Lambda(n) c_f(n)}{\sqrt{n}} \widehat \Phi\left(\frac{\log n}{\log Q} \right) = \frac{1}{\log Q}\sum_{p\nmid q} \frac{\lambda_f(p) \log p}{\sqrt{p}} \widehat \Phi\left(\frac{\log p}{\log Q} \right)\\
-\frac{1}{\log Q}\sum_{p\nmid q} \frac{ \log p}{p} \widehat \Phi\left(\frac{\log p^2}{\log Q} \right) 
+O\left(\frac{\log \log Q}{\log Q}\right).
\end{align*}

\end{lem}
We apply the prime number theorem and partial summation to the second sum over $p$ in the previous lemma and derive the following result.
\begin{lem}\label{lem:sump}
Let $q \sim Q$, where $Q \geq 20.$ We have
\begin{align*}
\frac{1}{\log Q}\sum_{p\nmid q} \frac{\log p}{p} \widehat \Phi\left(\frac{\log p^2}{\log Q} \right)  = \frac{1}{4} \Phi(0) + O\left(\frac{\log \log Q}{\log Q}\right).
\end{align*}

\end{lem}

Proposition \ref{prop:main} follows from the following two propositions. 
\begin{prop} \label{prop:boundSigma_1} Assume GRH. Let $k \geq 4$ and $\Phi$ be an even Schwartz function with $\widehat \Phi$ compactly supported in $(-4, 4).$  Moreover, define 
$$ \Sigma_1^{\pm} := \frac{1}{N^{\pm}(Q)}\sumb_q \Psi\bfrac{q}{Q} \frac{1}{\log Q}\sum_{f \in \mathcal H_{k}(q)}\sum_{p\nmid q} \frac{\lambda_f(p) \log p}{\sqrt{p}} \widehat \Phi\left(\frac{\log p}{\log Q} \right).$$
Then 
$$ \Sigma_1^{\pm} \ll \frac{1}{\log Q}. $$

\end{prop}

\begin{prop} \label{prop:boundSigma_2}
     Assume GRH. Let $k \geq 4$ and $\Phi$ be an even Schwartz function with $\widehat \Phi$ compactly supported in $(-3, 3).$  Moreover, define 
\est{ \Sigma_2^{\pm} &:= \frac{1}{N^{\pm}(Q)}\sumb_q \Psi\bfrac{q}{Q} \frac{1}{\log Q}\sum_{f \in \mathcal H_{k}(q)}\varepsilon_{f, q}\sum_{p\nmid q} \frac{\lambda_f(p) \log p}{\sqrt{p}} \widehat \Phi\left(\frac{\log p}{\log Q} \right). \\
&= \frac{i^k}{N^{\pm}(Q)}\sumb_q \mu(q) \sqrt q\Psi\bfrac{q}{Q} \frac{1}{\log Q}\sum_{f \in \mathcal H_{k}(q)}\sum_{p\nmid q} \frac{\lambda_f(pq) \log p}{\sqrt{p}} \widehat \Phi\left(\frac{\log p}{\log Q} \right).}
Then 
$$ \Sigma_2^{\pm} =  \frac{1}{2} \int_{|x| > 1} \widehat \Phi(x) \> dx + O\left( \frac{1}{\log Q}\right). $$
\end{prop}
We note that the stronger condition of $\widehat{\Phi}$ being supported on $(-3, 3)$ is needed when studying $\Sigma_{2}^{\pm}$ in Proposition \ref{prop:boundSigma_2}, while Proposition \ref{prop:boundSigma_1} can be proven with $\widehat{\Phi}$ being supported in the wider range $(-4, 4)$.

\subsection{Proof of Proposition \ref{prop:main} and Theorem \ref{thm:main}} \label{sec:proofofMainTheorem}
From Lemmas \ref{lem:asympforNQ}, \ref{lem:trimpsum}, \ref{lem:sump}, and Proposition~\ref{prop:boundSigma_1}, we have
\est{
&\frac{1}{N^{\pm}(Q) \log Q}  \sumb_q \Psi\bfrac{q}{Q}  \sum_{f \in \mathcal H_{k}(q)} \sum_n \frac{\Lambda(n)c_f(n)}{\sqrt{n}} \widehat \Phi\left( \frac{\log n}{\log Q}\right) \\
&\qquad =  -\frac{\Phi(0)}{2} + O\left( \frac{\log \log Q}{\log Q}\right) = -\frac{1}{2} \int_{-\infty}^{\infty} \widehat \Phi(x) \> dx + O\left( \frac{\log \log Q}{\log Q}\right).  
}
Taking $Q \rightarrow \infty$, we obtain \eqref{prop:withoutepsf} in Proposition~\ref{prop:main}. 

For \eqref{prop:withepsf}, from Lemma \ref{lem:sump} and Equation \ref{lem:asymDeltaq*}, we have
\est{& \frac{1}{N^{\pm}(Q)}\sumb_q \Psi\bfrac{q}{Q} \sum_{f \in \mathcal H_{k}(q)}\varepsilon_{f, q}  \frac{1}{\log Q}  \sum_{p \nmid q} \frac{\log p}{p} \widehat \Phi\left( \frac{\log p^2}{\log Q}\right) \\
&\qquad = \frac{i^{k} \Phi(0)}{4N^{\pm}(Q) }\sumb_q \Psi\bfrac{q}{Q}  \mu(q) \sqrt q\sum_{f \in \mathcal H_{k}(q)}\lambda_f(q) + O\left( \frac{\log \log Q}{\log Q}\right) \\
&\qquad \ll Q^{-1/6 + \epsilon} + \frac{\log \log Q}{\log Q}. }
By Lemma \ref{lem:trimpsum}, Proposition \ref{prop:boundSigma_2} and the above equation, we obtain \eqref{prop:withepsf} in Proposition \ref{prop:main}. 

To prove Theorem~\ref{thm:main}, we use the explicit formula from Lemmas~\ref{lem:Explicit} and \ref{lem:asympforNQ} and Proposition~\ref{prop:main} to deduce that
\begin{align*}
    &\lim_{Q \rightarrow \infty} \mathscr {OL}^{\pm} (Q) \\  &= - \lim_{Q \rightarrow \infty}\frac{1}{N^{\pm}(Q) \log Q}\sum_{q} \Psi\bfrac{q}{Q} \sum_{f \in \mathcal H_{k}(q) } \left( \frac{1 \pm \varepsilon_f}{2} \right) \sum_{n=1}^{\infty}\frac{\Lambda(n)[c_f(n) + c_{\bar f}(n)]}{\sqrt{n}} \widehat\Phi\left(\frac{\log n}{\log Q}\right) \\
    &+ \lim_{Q \rightarrow \infty} \frac{1}{N^{\pm}(Q)}\sum_{q} \Psi\bfrac{q}{Q}  \sum_{f \in \mathcal H_{k}(q)} \left( \frac{1 \pm \varepsilon_f}{2} \right)\left[\int_{-\infty}^{\infty} \Phi(x) \> dx +O_k\left(\frac{1}{\log Q} \right) \right], \\
    &= \frac{ 1}{2} \int_{-\infty}^{\infty} \widehat \Phi(x) \> dx  \mp  \frac{1}{2} \int_{|x| > 1} \widehat \Phi(x) \> dx   + \int_{-\infty}^{\infty} \Phi(x) \> dx.
\end{align*}
This proves Theorem~\ref{thm:main}.

\section{Truncation of the sums $\Sigma_{1}^{\pm}$ and $\Sigma_{2}^{\pm}$} \label{sec:truncationSigma1and2}

Applying Petersson's trace formula, Lemma~\ref{lem:PeterssonILS} and Lemma \ref{lem:truncLandm} to $\Sigma_{1}^{\pm}$ and $\Sigma_2^{\pm}$ defined in Proposition~\ref{prop:boundSigma_1} and \ref{prop:boundSigma_2}, respectively, we deduce that for $i = 1, 2$
\es{ \label{eqn:Sigmaafterhecke} 
\Sigma_i^{\pm} = \Sigma_{i, '}^{\pm} + \Sigma_{i, \infty}^{\pm},}
where
\es{ \label{def:Sigmai'}
 \Sigma_{i, '}^{\pm} &= \frac{k-1}{12N^{\pm}(Q)\log Q}\sumb_q \Psi\left(\frac{q}{Q}\right) B_{i, q}  \sum_{p\nmid q} \frac{ \log p}{\sqrt{p}} \widehat \Phi\left(\frac{\log p}{\log Q} \right) 
\\
& \qquad \times \sum_{\substack{LM = q \\ (L, M) = 1 \\ L \leq X}} \mu(L) M \sum_{\substack{(m, M) = 1 \\ m \leq Y}} \frac{\Delta_{M}(m^2,p t_{i, q})}{m}}
with $B_{1, q} = 1$, $t_{1, q} = 1$, $B_{2, q} = i^{k}\mu(q)\sqrt q$, $t_{2, q} = q$,
and 
\es{\label{def:Sigmaiinfty}  \Sigma_{i, \infty}^{\pm } = \frac{1}{N^{\pm}(Q)\log Q}\sumb_q \Psi\left(\frac{q}{Q}\right) B_{i, q}   \sum_{p\nmid q} \frac{ \log p }{\sqrt{p}} \widehat \Phi\left(\frac{\log p}{\log Q} \right) \Delta_q^{\infty}(pt_{i, q}). }

Let $a_j = \frac{\log j}{\sqrt j} \widehat \Phi\left(\frac{\log j}{\log Q} \right) $ if $j$ is prime and 0 otherwise. From Lemma \ref{lem:CLee3.5}, the sequence $(a_j)$ satisfies the property \eqref{propertysequenceaj}, and thus
$$\Sigma_{i, \infty}^{\pm }\ll_{k, \epsilon} (X^{-1} + Y^{-1}) (QXY)^{\epsilon/10} \ll Q^{-\epsilon/2}$$
upon choosing $X, Y$ to be $Q^{\epsilon}$. 

From now on, we focus on bounding $\Sigma_{i, '}^{\pm}.$ For $i=1$, we write

\est{  
\Sigma_{1, '}^{\pm}
&= \frac{k-1}{12N^{\pm}(Q)\log Q}\sum_q  \sum_{a^2 | q} \mu(a) \Psi\left(\frac{q}{Q}\right)   \sum_{p\nmid q} \frac{ \log p}{\sqrt{p}} \widehat \Phi\left(\frac{\log p}{\log Q} \right) \\
& \ \ \ \times \sum_{\substack{LM = q \\ (L, M) = 1 \\ L \leq Q^{\epsilon}}} \mu(L) M \sum_{\substack{(m, M) = 1 \\ m \leq Q^{\epsilon}}} \frac{\Delta_{M}(m^2,p)}{m}\\
&= \frac{k-1}{12 N^{\pm}(Q) \log Q}\sum_{\substack{a \\ (L, a) = 1}} \mu(a) \sum_{\substack{L, M \\ (L, M) = 1 \\ L \leq Q^{\epsilon}}}  \Psi\left(\frac{L Ma^2}{Q}\right)  \mu(L)Ma^2    \\
& \ \ \ \times \sum_{p\nmid LMa^2} \frac{ \log p}{\sqrt{p}} \widehat \Phi\left(\frac{\log p}{\log Q} \right) \sum_{\substack{(m, Ma) = 1 \\ m \leq Q^{\epsilon} }}\frac{\Delta_{Ma^2}(m^2,p)}{m}.}

Next, we truncate the sum over $a$ and obtain the following result.
\begin{lem} \label{lem:truncatea} We have
    $$\Sigma_{1,'}^{\pm} = \Sigma_{1, main}^{\pm} + O(Q^{-9\epsilon/10}),$$
    where $\Sigma_{1, main}^{\pm}$ is defined analogously to $\Sigma_{1, '}^{\pm}$ but with $a \leq Q^{\epsilon}.$
\end{lem}
\begin{proof}
We would like to bound 
\begin{align*}
&\ll \frac{1}{Q^2 \log Q} \sum_{a \geq A}  \sum_{\substack{L, M \\ L \leq Q^{\epsilon}}} \Psi\left( \frac{LMa^2}{Q}\right)  Ma^2 
\sum_{\substack{m \leq Q^{\epsilon}  }} \frac{\tau(m^2)}{m}  \left| \sumh_{f\in B_k(Ma^2)} \sum_{p \nmid LMa^2} \frac{ \log p \lambda_f(p)} {\sqrt{p}} \widehat \Phi\left(\frac{\log p}{\log Q} \right)\right|  \\
&\ll \frac{1}{Q^2 \log Q} \sum_{a \geq A}  \sum_{\substack{L, M \\ L \leq Q^{\epsilon}}} \Psi\left( \frac{LMa^2}{Q}\right)  Ma^2  
\sum_{\substack{m \leq Q^{\epsilon}  }} \frac{\tau(m^2)}{m} \log^{2 + \epsilon} q,
\end{align*}
upon applying Lemma~\ref{lem:CLee3.5} to the sum over $p$.  We remind the reader that we have used the fact that $\lambda_f(p) = \lambda_g(p)$ for some Hecke newform $g$ of level dividing $Ma^2$.  The quantity above is bounded by
\begin{align*}
&\ll \frac{Q^{\epsilon/10}}{Q}  \sum_{a \geq A} \sum_{\substack{L, M \\ L \leq q^{\epsilon}}} \Psi\left( \frac{LMa^2}{Q}\right)  \frac{1}{L}  \ll Q^{\epsilon/10}  \sum_{a \geq A} \frac{1}{a^2} \sum_{L }  \frac{1}{L^2} \ll \frac{Q^{\epsilon/10}}{A} \ll Q^{-9\epsilon/10}
\end{align*} 
upon choosing $A = Q^{\epsilon}$. 

\end{proof}

\section{Proof of Proposition \ref{prop:boundSigma_1} - Bounding $\Sigma_{1, main}^{\pm}$}

By Lemma \ref{lem:truncatea}, Proposition \ref{prop:boundSigma_1} immediately follows from the following. 
\begin{lem} \label{lem:boundforSigma1main} Let notation be as above. Then 
$$ \Sigma_{1, main}^{\pm} \ll \frac{1}{\log Q}.
$$    
\end{lem}

Since $a, L, m \le Q^\epsilon$, it suffices to show that for fixed $a, L, m$ we have

\begin{align}\label{eqn:Sigma1main1}
\frac{1}{Q^2} \sum_{\substack{M \\ (Lm, M) = 1}}  \Psi\left(\frac{L Ma^2}{Q}\right)  M  \sum_{p\nmid LMa^2} \frac{ \log p}{\sqrt{p}} \widehat \Phi\left(\frac{\log p}{\log Q} \right) \Delta_{Ma^2}(m^2,p) \ll Q^{- \delta_0}
\end{align}
for some $\delta_0 >0$.  The steps of this are essentially the same as in the work \cite{BCL}.  First, following the same sort of calculation as in \S 5.3 of \cite{BCL}, we may drop the condition $p \nmid LMa^2$ with an error of $\ll \frac{1}{Q^{1-\epsilon}}$.  To be precise, since $p\neq m^2$, Lemma~\ref{lem:petertruncate} implies that
\begin{align*}
\Delta_{Ma^2}(m^2,p)
& \ll  \frac{(pmMa)^{\epsilon} }{ M a^2}(pm^2)^{1/4},
\end{align*}
and since $aLm \ll Q^\epsilon$ the contribution of those $p|LMa^2$ is bounded by

\begin{align*}
    \frac{Q^\epsilon}{Q^2} \sum_{\substack{M }}  \Psi\left(\frac{L Ma^2}{Q}\right)    \sum_{p|LMa^2} \frac{ \log p}{p^{1/4}}  \left|\widehat{\Phi}\left(\frac{\log p}{\log Q} \right)\right| \ll Q^{-1+\epsilon}.
\end{align*}

This reduces the desired bound in \eqref{eqn:Sigma1main1} to proving that
\begin{align*}
     \frac{1}{Q^2} \sum_{p} \frac{ \log p}{\sqrt{p}} \widehat \Phi\left(\frac{\log p}{\log Q} \right)  \sum_{\substack{M \\ (Lm, M) = 1}}  \Psi\left(\frac{L Ma^2}{Q}\right)  M  \Delta_{Ma^2}(m^2,p) \ll Q^{- \delta_0}
\end{align*}
for some $\delta_0 >0$.  We further express the coprimality condition $(Lm, M) = 1$ using M\"obius inversion so that the quantity above is

\begin{align*}
    \frac{1}{Q^2} \sum_{p} \frac{ \log p}{\sqrt{p}} \widehat \Phi\left(\frac{\log p}{\log Q} \right) \sum_{\ell | Lm} \mu(\ell) \ell  \sum_{M}  \Psi\left(\frac{\ell L Ma^2}{Q}\right)  M  \Delta_{\ell a^2 M}(m^2,p),
\end{align*}
but since $Lm \ll Q^\epsilon$ it suffices to show that 

\begin{align}\label{eqn:frakSbdd}
    \mathfrak S :=  \frac{1}{Q^2} \sum_{p} \frac{ \log p}{\sqrt{p}} \widehat \Phi\left(\frac{\log p}{\log Q} \right)   \sum_{M}  \Psi\left(\frac{\ell L Ma^2}{Q}\right)  M  \Delta_{\ell a^2 M}(m^2,p) \ll Q^{-\delta_0}
\end{align}for some $\delta_0 >0$ and for fixed $a, L, \ell, m \ll Q^\epsilon$.  Since $\Psi$ is compactly supported on the positive real numbers, $M \asymp \frac{Q}{\ell L a^2}$, and writing $\mathfrak W(x) = \Psi(x) x$, we have
$$\Psi\left(\frac{\ell L Ma^2}{Q}\right)  M  = \frac{Q}{\ell L a^2} \mathfrak W\left(\frac{\ell L Ma^2}{Q}\right).
$$Note that $\mathfrak W$ is still a smooth, compactly supported function on the positive real line.  Further, since $p\neq m^2$ for any prime $p$ and integer $m$, we write 

\begin{align*}
    \mathfrak S &=  \frac{1}{Q\ell La^2} \sum_{p} \frac{ \log p}{\sqrt{p}} \widehat \Phi\left(\frac{\log p}{\log Q} \right)   \sum_{M}  \mathfrak W\left(\frac{\ell L Ma^2}{Q}\right)   \sum_{c\ge 1} \frac{S(m^2, p; c\ell a^2M)}{c\ell a^2 M} J_{k-1} \bfrac{4\pi m \sqrt{p}}{c\ell a^2 M} \\
    &=  \frac{1}{Q\ell La^2} \sum_{c\ge 1} \sumd_P \sum_{p} \frac{ \log p}{\sqrt{p}} \widehat \Phi\left(\frac{\log p}{\log Q} \right) V\bfrac{p}{P}  \sum_{M}  \frac{S(m^2, p; c\ell a^2M)}{c\ell a^2 M} f \bfrac{4\pi \sqrt{pm^2}}{c\ell a^2M},
\end{align*}
where 
\begin{align}\label{eqn:fdef}
    f(\xi) = \mathfrak W\bfrac{4\pi L \sqrt{pm^2}}{cQ \xi} J_{k-1}(\xi),
\end{align}
and we have introduced a smooth partition of unity.  To be more precise, in the statement of  Proposition \ref{prop:boundSigma_1}, we have assumed that $\widehat \Phi$ is supported on $(-4, 4)$, so there exists some $\delta > 0$ such that $\widehat \Phi(x) = 0$ for $|x| > 4-\delta$.  Then we let $\sumd_{\!\!\!P}$ denote a sum over $P = 2^j \le Q^{4 - \delta}$ for $j\ge 0$ and where $V$ is a smooth function compactly supported on $[1/2, 3]$ satisfying $\sumd_{\!\!\!P} V\bfrac{p}{P} = 1$ for all $1\le p \le Q^{4-\delta}$.

It will be convenient for us to extract the dependence of $f(\xi)$ on $p$, replacing it by dependence on $P$ instead.  To this end, let
\begin{equation}
    X = \frac{4\pi L \sqrt{Pm^2}}{cQ},
\end{equation}
and thinking of $\lambda = \frac{p}{P}$,
\begin{equation}\label{eqn:Hdef}
    H(\xi, \lambda) := \mathfrak W\bfrac{X \sqrt{\lambda}}{\xi} \hat{\Phi}\bfrac{\log \lambda+ \log P}{\log Q} V_0(\lambda),
\end{equation}
where we have introduced a smooth function $V_0$ which is compactly supported on the positive real line satisfying $V_0(\lambda) = 1$ when $\lambda \in [1/2, 3]$.  Note that this implies that $V\bfrac{p}{P} V_0\bfrac{p}{P} = V\bfrac{p}{P}$ by the support condition on $V$.  We have inserted $V_0$ into the definition of $H$ simply to record the restriction $\lambda \asymp 1$, which is a condition used in Lemma \ref{lem:hatHbdd} below.

Let 
$$\widehat{H}(u, v) = \int_{-\infty}^\infty \int_{-\infty}^\infty H(\xi, \lambda) \e{-\xi u - \lambda v} d\xi d\lambda$$ be the usual Fourier transform of $H$.  We record the following bounds on $\widehat H$ which is essentially Lemma 6.1 of \cite{BCL}.  Our function $H$ is slightly different, but the proof is similar.

\begin{lem}\label{lem:hatHbdd}
With notation as above, we have for $\lambda\asymp 1$ that for any $A_1, A_2 \ge 0$, 
\begin{equation*}
\widehat{H}(u, v) \ll_{A_1, A_2} \frac{X}{(1 + |u|X)^{A_1}} \frac{1}{(1+|v|)^{A_2}}.
\end{equation*}
\end{lem}

Note that $H(\xi, \frac pP)$ is compactly supported in $\xi$ by the support condition on $\mathfrak W$.  Let us say it is supported in $(a_1, b_1)$.  By \eqref{eqn:fdef} and \eqref{eqn:Hdef}, we have that
\begin{align*}
f(\xi) &= H\left(\xi, \frac pP\right) J_{k-1}(\xi) \\\
&= H\left(\xi, \frac pP\right) J_{k-1}(\xi) \, W\left( \frac{\xi}{X} \right) \\
&=  \int_{-\infty}^\infty  \int_{-\infty}^\infty  \widehat{H}(u, v) e(u\xi) e\left(\frac{p}{P} v\right) \>du \>dv  \> J_{k-1}(\xi) \,W\left( \frac{\xi}{X} \right), 
\end{align*}
where $W$ is a smooth function that is compactly supported in $(a, b), $ $0<a<a_1$, $b_1 < b$, such that $W(\xi) = 1$ when $\xi \in [a_1, b_1]$, and the last line follows by Fourier inversion. Hence,

\begin{align*}
    \mathfrak S =  \frac{1}{Q\ell La^2} \sumd_P \sum_{c\ge 1}  \int_{-\infty}^\infty  \int_{-\infty}^\infty \widehat H(u, v) \sum_{p} \frac{ \log p}{\sqrt{p}} V\bfrac{p}{P}  e\left(\frac{p}{P} v\right) \mathcal\cS(c, p ; u) \>du \>dv ,
\end{align*}
where 
\begin{align*}
    \cS(c, p ; u) = \sum_{M}  \frac{S(m^2, p; c\ell a^2M)}{c\ell a^2 M} h_u\bfrac{4\pi \sqrt{pm^2}}{c\ell a^2M}
\end{align*}
with 
\begin{align*}
    h_u(\xi) = J_{k-1}(\xi) e(u\xi) W\bfrac{\xi}{X}.
\end{align*}

We apply Kuznetsov's formula from Lemma~\ref{lem:kuznetsov} to  $\mathcal S(u, p)$ and obtain that
\es{ \label{eqn:sumoverpseparateintoDisCtnHol}
\cS(c, p ; u) = \textup{Dis}(c,p; u) + \Ctn(c,p; u) + \Hol(c,p; u);
}
where

\begin{align*}
\textup{Dis}(c,p; u) &:=  \sum_{j=1}^{\infty} \frac{\overline{\rho_j}(m^2)\rho_j(p) \sqrt{pm^2} }{\cosh(\pi \kappa_j)} h_+(\kappa_j), \\
\Ctn(c,p; u) &:= \frac{1}{\pi} \sum_{\mathfrak{c}} \int_{-\infty}^{\infty}  \frac{\sqrt{pm^2}}{\cosh(\pi t)} \overline{\varphi_{\mathfrak{c}}} (m^2, t)  \varphi_{\mathfrak{c} } (p, t) h_+ (t) \,dt \textup{, \;\;\;and}\\
\Hol(c,p; u) &:= \frac{1}{2\pi} \sum_{ \substack{r\geq 2 \mbox{\scriptsize{ even}} \\ 1 \leq j \leq \theta_{r}(c\ell a^2)} } (r - 1)! \sqrt{pm^2} \, \overline{\psi_{j,r}}(m^2) \psi_{j,r} (p) h_h(r),
\end{align*}where it is implied that $a, L, 
\ell, m$ are fixed. Note that these forms are of level $c\ell a^2.$  We write

\es{\label{def:M_DCH} \mathfrak M_{A} (P) = \sum_{c\geq 1} \int_{-\infty}^\infty \int_{-\infty}^\infty    \widehat{H}(u, v)\sum_{p } \frac{ \log p}{\sqrt{p}} \e{v\frac{p}{P}} V\bfrac{p}{P} A(c, p; u) \> du \> dv, }
where $A(c, p; u)$ is $\textup{Dis}(c,p; u)$, $\Ctn(c, p; u)$, or $\Hol(c, p; u)$. Hence, 
\begin{align}\label{eqn:frakSfinalexpression}
    \mathfrak S = \frac{1}{Q\ell La^2} \sumd_P  \left(\mathfrak M_{\textup{Dis}} (P) + \mathfrak M_{\Ctn} (P) + \mathfrak M_{\Hol} (P)\right).
\end{align}

The required bound for $\mathfrak S$ will follow from the following propositions.

\begin{prop}\label{prop:DisHol}
With notation as above, we have
\begin{align*}
 \mathfrak M_{\Dis}(P)\ll Q^{\epsilon} \frac{\sqrt P}{Q} \ \ \ \ \textrm{and} \ \ \ \ \ \ \mathfrak M_{\Hol}(P)\ll Q^{\epsilon} \frac{\sqrt P}{Q}.
\end{align*}
\end{prop}

\begin{prop}\label{prop:Ctn}
With notation as above, we have
\begin{align*}
\mathfrak M_{\Ctn}(P) \ll Q^{\epsilon} \left(P^{\frac 14 + \epsilon} + \frac{\sqrt P}{Q}\right).
\end{align*}
\end{prop}

Proposition \ref{prop:DisHol} is \cite[Proposition 6.2]{BCL} and Proposition \ref{prop:Ctn} is \cite[Proposition 6.3]{BCL}, and we refer the reader to \cite{BCL} for the proofs.

We now directly verify \eqref{eqn:frakSbdd} using Propositions \ref{prop:DisHol} and \ref{prop:Ctn}.  Indeed, plugging in the bounds from Propositions \ref{prop:DisHol} and \ref{prop:Ctn} into \eqref{eqn:frakSfinalexpression} we have

\begin{align*}
    \mathfrak S 
    &\ll  \frac{Q^{\epsilon}}{Q\ell La^2} \sumd_P \left(  \frac{\sqrt P}{Q} +P^{\frac 14 + \epsilon}\right) 
    \ll  \frac{Q^{\epsilon}}{Q\ell La^2} \left(  Q^{1- \delta/2 } + Q^{1 - \delta/4}\right) 
     \ll Q^{- \delta/4 + \epsilon},
\end{align*}
where we have used that $P \le Q^{4 - \delta}$, and this suffices for \eqref{eqn:frakSbdd}.

\section{Setup for computing $\Sigma_{2,'}^{\pm}$}

By \eqref{def:Sigmai'}, we have
\est{
 \Sigma_{2, '}^{\pm} 
&= \frac{(k-1)i^{k}}{12N^{\pm}(Q)\log Q} \sum_{\substack{L, M \\ (L, M) = 1 \\ L \leq Q^{\epsilon}}} \mu^2(L) \mu(M) M \sqrt{LM} \Psi\left( \frac{LM}{Q}\right)    \\
&\qquad \times \sum_{p\nmid LM} \frac{ \log p}{\sqrt{p}} \widehat \Phi\left(\frac{\log p}{\log Q} \right) \sum_{\substack{(m, M) = 1 \\ m \leq Q^{\epsilon}}} \frac{\Delta_{M}(m^2,pLM)}{m}. }

Since $p, L, M$ are pairwise relatively prime, $pLM \neq m^2$ for any positive integers $m$. Applying the Petersson trace formula in Lemma \ref{lem:usualPetersson} and the definition of the Kloosterman sum, we obtain that

\est{\Sigma_{2,'}^{\pm} &= \frac{2\pi (k-1)}{12N^{\pm}(Q)\log Q} \sum_{\substack{L, M \\ (L, M) = 1 \\ L \leq Q^{\epsilon}}} \mu^2(L) \mu(M) M \sqrt{LM} \Psi\left( \frac{LM}{Q}\right)    \sum_{p\nmid LM} \frac{ \log p}{\sqrt{p}} \widehat \Phi\left(\frac{\log p}{\log Q} \right) \\
& \qquad \times \sum_{\substack{(m, M) = 1 \\ m \leq Q^{\epsilon}}} \frac{1}{m} \sum_{c} \frac{1}{cM} \sumstar_{b \mod {cM}} \e{\frac{m^2b}{cM}}  \e{\frac{\bar b pL}{c}} J_{k-1}\left( \frac{4\pi m\sqrt{pL}}{c\sqrt M}\right).}

We now add the condition $(p, c) = 1$ to facilitate the later introduction of Dirichlet characters. 

\begin{lem} We have
$$ \Sigma_{2, '}^{\pm} = \Sigma_{3}^{\pm} + O(Q^{-\epsilon}),$$
where $\Sigma_{3}^{\pm}$ is defined analogously to $\Sigma_{2, '}^{\pm}$ with the extra condition that $(p, c) = 1.$
\end{lem}

\begin{proof} To prove the lemma, we need to bound
\est{\mathcal E^{\pm}_{2, p | c}  &:= \frac{2\pi (k-1)}{12N^{\pm}(Q)\log Q} \sum_{\substack{L, M \\ (L, M) = 1 \\ L \leq Q^{\epsilon}}} \mu^2(L) \mu(M) M \sqrt{LM} \Psi\left( \frac{LM}{Q}\right)    \sum_{p\nmid LM} \frac{ \log p}{\sqrt{p}} \widehat \Phi\left(\frac{\log p}{\log Q} \right) \\
& \qquad \times \sum_{\substack{(m, M) = 1 \\ m \leq Q^{\epsilon}}} \frac{1}{m} \sum_{\substack{c \\ p | c}} \frac{S(m^2, pLM, cM)}{cM}  J_{k-1}\left( \frac{4\pi m\sqrt{pL}}{c\sqrt M}\right). }
Weil's bound for the Kloosterman sum gives that
$$ S(m^2, pLM, pcM) \ll (cpM)^{1/2 + \epsilon} m \ll Q^\epsilon (cpM)^{1/2 + \epsilon}.$$
Changing variables from $c$ to $pc$, and putting in the bound for the Bessel function from Lemma \ref{jbessel}, we obtain that

\est{\mathcal E^{\pm}_{2, p | c}  &\ll  \frac{1}{Q^2} \sum_{\substack{L, M \\ (L, M) = 1 \\ L \leq Q^{\epsilon}}} M \sqrt{LM} \Psi\left( \frac{LM}{Q}\right)    \sum_{p\nmid LM} \frac{ \log p}{\sqrt{p}} \left|\widehat \Phi\left(\frac{\log p}{\log Q} \right)\right| \\
&\ \ \ \ \times \sum_{\substack{(m, M) = 1 \\ m \leq Q^{\epsilon}}} \frac{1}{m} \sum_{\substack{c }} (cpM)^{-1/2 + \epsilon}  \left( \frac{m\sqrt{L}}{c\sqrt{pM}}\right)^{k - 1}. }
For $k \geq 4,$ the sums over $c$ and $p$ are absolutely convergent and we have 
\est{ \mathcal E_{2, p | c}^{\pm} &\ll \frac{Q^{\epsilon}}{Q^2}\sum_{\substack{L, M \\ (L, M) = 1 \\ L \leq Q^{\epsilon}}} M \sqrt{L} \Psi\left( \frac{LM}{Q}\right) \sum_{\substack{(m, M) = 1 \\ m \leq Q^{\epsilon}}} \frac{1}{m} \left( \frac{m\sqrt L}{\sqrt M}\right)^{k - 1}  \\
&\ll \frac{Q^{\epsilon k}}{Q^2} \frac{1}{Q^{k/2 - 5/2 }} \ll \frac{Q^{\epsilon}}{Q^{k/2 - 1/2}} \ll Q^{-\epsilon}.}
This proves the lemma.
\end{proof}
Next we remove the condition $p \nmid M$, decoupling the sum over $p$ and the sum over $M$. This will be convenient when bounding the sums over $p$ and $m$. 

\begin{lem} We have
$$ \Sigma_{3}^{\pm} = \Sigma_{4}^{\pm} + O(Q^{-3/2+\epsilon}), $$
    where $\Sigma^{\pm}_4$ is defined analogously to $\Sigma_{3}^{\pm}$ with the condition $p \nmid Lc$ instead of $p \nmid LMc.$
\end{lem}

\begin{proof} To prove the lemma, we need to bound
\est{\mathcal E^{\pm}_{3, p | M}  &= \frac{2\pi (k-1)}{12N^{\pm}(Q)\log Q} \sum_{\substack{L, M \\ (L, M) = 1 \\ L \leq Q^{\epsilon}}} \mu^2(L) \mu(M) M \sqrt{LM} \Psi\left( \frac{LM}{Q}\right)    \sum_{\substack{p\nmid L \\ p | M}} \frac{ \log p}{\sqrt{p}} \widehat \Phi\left(\frac{\log p}{\log Q} \right) \\
& \qquad \times \sum_{\substack{(m, M) = 1 \\ m \leq Q^{\epsilon}}} \frac{1}{m} \sum_{\substack{(c, p) = 1}} \frac{S(m^2, pLM, cM)}{cM}  J_{k-1}\left( \frac{4\pi m\sqrt{pL}}{c\sqrt M}\right). }
    We change the variable from $M$ to $pM$, and so $(p, m) = 1$. From Weil's bound for the Kloosterman sum, we have
    \est{ S(m^2, p^2LM, cpM) &\ll (cpM)^{1/2 + \epsilon} \sqrt{(m^2, p^2LM, cpM)} \\ 
    &\ll (cpM)^{1/2+\epsilon} m \ll (cpM)^{1/2 + \epsilon}Q^\epsilon, } for $m\le Q^\epsilon$.
    This, and Lemma \ref{jbessel} gives us that
    \est{\mathcal E^{\pm}_{3, p | M}  &\ll  \frac{Q^\epsilon}{Q^2} \sum_{\substack{L, M \\ L \leq Q^{\epsilon}}} M \sqrt{LM}    \sum_{p\nmid L} p \log p \Psi\left( \frac{pLM}{Q}\right) \left|\widehat \Phi\left(\frac{\log p}{\log Q} \right)\right| \\
& \qquad \times \sum_{\substack{(m, pM) = 1 \\ m \leq Q^{\epsilon}}} \frac{1}{m} \sum_{\substack{(c, p) = 1 }} (cpM)^{-1/2 + \epsilon} \min \left\{ \left( \frac{m\sqrt{L}}{c\sqrt{M}}\right)^{-1/2}, \left( \frac{m\sqrt{L}}{c\sqrt{M}}\right)^{k - 1} \right\} \\
&\ll \frac{Q^\epsilon}{Q^2} \sum_{\substack{L, M \\ L \leq Q^{\epsilon}}} M^{1-(k-1)/2}    \sum_{p\nmid L} p^{1/2} \Psi\left( \frac{pLM}{Q}\right) \sum_{c\ge 1} c^{-1/2 + \epsilon} \left( \frac{1}{c}\right)^{k - 1}\\
&\ll \frac{Q^\epsilon}{Q^2} \sum_{M\ll Q} M^{1-(k-1)/2}    \bfrac{Q}{M}^{1/2} \\
&\ll Q^{-3/2+\epsilon},}
for $k \ge 4$, as desired.
\end{proof}

Now, we write $e\bfrac{\bar b p L}{c} = e\bfrac{\bar b d L}{c}$ for $p \equiv d \bmod c$, and express the latter congruence using Dirichlet characters. Specifically,
\es{ \label{eqn:separatetotrivandnonchar}\Sigma_{4}^{\pm} &= \frac{2\pi (k-1)}{12N^{\pm}(Q)\log Q} \sum_{\substack{L, M \\ (L, M) = 1 \\ L \leq Q^{\epsilon}}} \mu^2(L) \mu(M) M \sqrt{LM} \Psi\left( \frac{LM}{Q}\right)  \sum_{\substack{(m, M) = 1 \\ m \leq Q^{\epsilon}}} \frac{1}{m}   \\
&\qquad \times  \sum_{c} \frac{1}{cM} \sumstar_{b \mod {cM}} \sumstar_{d\mod c} \e{\frac{m^2b}{cM}}  \e{\frac{\bar b dL}{c}} \frac{1}{\phi(c)} \sum_{\chi \mod c}\overline{\chi(d)} \\
& \qquad \times
\sum_{p\nmid L} \frac{ \log p}{\sqrt{p}} \chi(p)\widehat \Phi\left(\frac{\log p}{\log Q} \right)J_{k-1}\left( \frac{4\pi m\sqrt{pL}}{c\sqrt M}\right) \\
&=: \Sigma_{triv}^{\pm} + \Sigma_{non}^{\pm},}
where $\Sigma_{triv}^{\pm}$ is the contribution from the trivial character $\chi_0$ mod $c$, and $\Sigma_{non}^{\pm}$ is the contribution from the non-trivial characters.  Proposition \ref{prop:boundSigma_2} immediately follows from the lemmas below. 

\begin{lem} \label{lem:trivchar} Let notation be as above. Assume GRH and let $\Phi$ be an even Schwartz function with $\widehat \Phi$ compactly supported in $(-3, 3).$ Then
\est{\Sigma_{triv}^{\pm} = \frac{1}{2} \int_{|y| > 1} \widehat \Phi(y) \> dy + O\left( \frac{1}{\log Q}\right).}
    
\end{lem}

\begin{lem} \label{lem:nontrivchar} Let notation be as above. Assume GRH and let $\Phi$ be an even Schwartz function with $\widehat \Phi$ compactly supported in $(-3, 3).$ Then
\est{\Sigma_{non}^{\pm} \ll_{k, \epsilon} Q^{-\epsilon}.}
\end{lem}

\section{Proof of Lemma \ref{lem:trivchar} - Trivial character contribution}
When $\chi$ is a trivial character, the sums over $d$ and $b$ in \eqref{eqn:separatetotrivandnonchar} are Ramanujan's sums. We therefore have

\est{ \Sigma_{triv}^{\pm} &=
\frac{2\pi (k-1)}{12N^{\pm}(Q)\log Q} \sum_{\substack{L, M \\ (L, M) = 1 \\ L \leq Q^{\epsilon}}} \mu^2(L) \mu(M) M \sqrt{LM} \Psi\left( \frac{LM}{Q}\right)  \sum_{\substack{(m, M) = 1 \\ m \leq Q^{\epsilon}}} \frac{1}{m}   \\
&\qquad \times  \sum_{c} \frac{\phi(cM)}{c M} \frac{\mu\left( \frac{cM}{(c, m^2)} \right)}{\phi\left(\frac{cM}{(c, m^2)}\right)} \frac{\mu\left( \frac{c}{(c,L)}\right)}{\phi\left( \frac{c}{(c,L)}\right)} 
\sum_{p\nmid cL} \frac{ \log p}{\sqrt{p}} \widehat \Phi\left(\frac{\log p}{\log Q} \right)J_{k-1}\left( \frac{4\pi m\sqrt{pL}}{c\sqrt M}\right).}
Assuming RH, by the prime number theorem 
$$\sum_{\substack{p \leq x \\ (p, \ell) = 1}} \log p = \sum_{\substack{p \leq x}} \log p - \sum_{p | \ell } \log p = x + O(x^{1/2}\log^2 x + \log \ell).$$
Hence, partial summation gives that
\es{\label{eqn:sumpfromPNT} \sum_{\substack{p \nmid cL}} &\frac{ \log p}{\sqrt{p}} \widehat \Phi\left(\frac{\log p}{\log Q} \right)J_{k-1}\left( \frac{4\pi m \sqrt{pd}}{c\sqrt M}\right) \\
&= \int_1^{\infty} \frac{1}{\sqrt t} \widehat \Phi\left( \frac{\log t}{\log Q}\right) J_{k-1}\left( \frac{4\pi m \sqrt{tL}}{c\sqrt d}\right) \> dt \\
&+ O\left(  \int_1^{\infty} (\sqrt t (\log t)^2 + \log (cL)) \left|\frac{d}{dt}\left[\frac{1}{\sqrt t} \widehat \Phi\left( \frac{\log t}{\log Q}\right) J_{k-1}\left( \frac{4\pi m \sqrt{tL}}{c\sqrt M}\right)\right]\right| \> dt  \right), }
and we can write
\es{\label{def:Sigmatriv_mainanderrorsump}
\Sigma_{triv}^{\pm} =: \Sigma_{triv, main}^{\pm} + \Sigma_{triv, err}^{\pm},
}
where $\Sigma_{triv, main}^{\pm} $ is the contribution from the main integral in Equation \eqref{eqn:sumpfromPNT}, and similarly $\Sigma_{triv, err}^{\pm}$ is in the contribution in the error term inside the big-$O$ in \eqref{eqn:sumpfromPNT}.

\subsection{Bounding $\Sigma_{triv, err}^{\pm}$}

\begin{lem}  Let $\Phi$ be an even Schwartz function with $\widehat \Phi$ compactly supported in $(-3, 3)$, and recall that the weight $k \geq 4.$ Then
$$ \Sigma_{triv, err}^{\pm} \ll Q^{-\epsilon}.$$
\end{lem}

\begin{proof}
    We note that $2J_v'(y) = J_{v-1}(y) - J_{v+1}(y) $ (e.g. see p.97 in \cite{ILS}).  Moreover, the support of $\widehat \Phi$ restricts $t$ to $t\le Q^{3-\delta}$ for some $\delta > 0$.  Thus 
\est{\Sigma_{triv, err}^{\pm} &\ll \frac{Q^{\epsilon}}{Q^2}\sum_{\substack{L, M \\  L \leq Q^{\epsilon}} } M^{1/2}L^{1/2} \,  \Psi\left(\frac{LM}{Q}\right) \sum_{\substack{ m \leq Q^{\epsilon}}} \frac{1}{m} \sum_c \frac{1}{c^2}  \\
&\qquad\times \int_{1}^{Q^{3-\delta}} \frac{\sqrt t + \log c}{t^{3/2}}\left( 1 + \frac{m\sqrt{tL}}{c\sqrt M}\right) \left|J_{v}\left( \frac{4\pi m \sqrt{tL}}{c\sqrt M}\right)\right| \> dt}
where $v \in \{k-2, k-1, k\}$ is such that it maximizes the quantity on the right above. For $t \ll \frac{c^2M}{m^2L}$, by Lemma \ref{jbessel}, we have $J_v(y) \ll y^{v}$, and this contributions 
\est{\Sigma_{triv, err, sm \ t}^{\pm} &\ll \frac{Q^{\epsilon}}{Q^2 }\sum_{\substack{L, M \\  L \leq Q^{\epsilon}} } M^{1/2}L^{1/2} \,  \Psi\left(\frac{LM}{Q}\right) \sum_{\substack{ m \leq Q^{\epsilon}}} \frac{1}{m} \sum_c \frac{\log c}{c^2}  \int_{1}^{\frac{c^2M}{m^2L}} \frac{1}{t}\left( \frac{ m \sqrt{tL}}{c\sqrt M}\right)^{v}  \> dt \\
&\ll \frac{Q^{\epsilon}}{\sqrt {Q}} \ll Q^{-\epsilon}.}
For $t \gg \frac{c^2M}{m^2 L}$, Lemma \ref{jbessel} gives $J_v(y) \ll y^{-1/2}$, and the contribution is

\est{\Sigma_{triv, err, big \ t}^{\pm} & \ll \frac{Q^{\epsilon}}{Q^2 }\sum_{\substack{L, M \\  L \leq Q^{\epsilon}} } M^{1/2}L^{1/2} \,  \Psi\left(\frac{LM}{Q}\right) \sum_{\substack{ m \leq Q^{\epsilon}}} \frac{1}{m} \sum_c \frac{\log c}{c^2}  \int_{\frac{c^2M}{m^2L}}^{Q^{3 - \delta}} \frac{1}{t}\left( \frac{ m \sqrt{tL}}{c\sqrt M}\right)^{1/2}  \> dt \\
&\ll \frac{Q^{3/4 - \delta/4 + \epsilon}}{Q^{2}} \sum_{\substack{L, M \sim \frac{Q}{L} \\ L \leq Q^{\epsilon}}} M^{1/4} L^{3/4} \ll Q^{-\delta/4 + \epsilon} \ll Q^{-\epsilon}}
when $\epsilon$ is small enough in terms of $\delta.$
\end{proof}

\subsection{Evaluating $\Sigma_{triv, main}^{\pm}$}
We now examine the main term, and complete the proof of Lemma \ref{lem:trivchar}. We begin with the following Lemma.
\begin{lem} \label{lem:sumMandsumc} Let $L$ be a square-free positive integer. Define
   \est{\mathcal S(L, m) &:= \sum_{(M, mL) = 1} \frac{\mu(M)}{M^{-1/2 + w/2 + s}}   \sum_c \frac{\phi(c M)}{c^{1 + w}} \frac{\mu\left( \frac{cM}{(c, m^2)} \right)}{\phi\left(\frac{cM}{(c, m^2)}\right)} \frac{\mu\left( \frac{c}{(c,L)}\right)}{\phi\left( \frac{c}{(c,L)}\right)} . } 
   Then 
   $$\mathcal S(L, m) = \zeta(2 + w) \zeta(-1/2 + w/2 + s) H(w, s; L, m),$$
    where 
    \begin{equation}
     H(w,s; L, m) = \prod_p H_p(w, s; L,m),
    \end{equation}
    for some $H_p(w, s, L, m)$ where the product expression for $H(w,s; L, m)$ above is absolutely convergent or is $0$ when $\tRe(w) > - 3/2$ and $\tRe(w + 2s) > 2.$  Moreover, 
    \begin{align} \label{Hp(-1, 2)}
H_p(-1, 2; L, m) = \left\{ \begin{array}{ll}
     1 - \frac{2}{p^2} + \frac{1}{p^3}  &  \ \textrm{if} \ \  p \nmid mL; \\
     0 &  \ \textrm{otherwise.} 
\end{array} \right. 
    \end{align}
    
\end{lem}

\begin{proof}
    Let $\mathcal S_p(L,m)$ be the Euler factor at $p$ of $\mathcal S(L, m)$. When $p \nmid mL$,  
    $$\mathcal S_p(L, m) = 1 + \frac{1}{p^{-1/2 + w/2 + s}} + \frac{1}{p^{1 + w}(p - 1)}.$$
When $p | mL$, 
\begin{align*} \mathcal S_p(L, m) =  \left\{
\begin{array}{ll}
     1 - \frac{1}{p^{1 + w}}  &  \ \textrm{if} \ \  p | m \ \textrm{and} \  p\nmid L \ \  \textrm{or}  \ \  p | L \ \textrm{and} \  p\nmid m; \\
     1 + \frac{p - 1}{p^{1 + w}} - \frac{p}{p^{2 + 2w}} &  \ \textrm{if} \ \  p | m \ \textrm{and} \ p | L. 
\end{array} \right. 
    \end{align*}
Then the lemma follows from above equations.  
\end{proof}

\begin{lem}  Let $\Phi$ be an even Schwartz function with $\widehat \Phi$ compactly supported in $(-3, 3)$, and also let the weight $k \geq 4.$ Then
\est{\Sigma_{triv, main}^{\pm} = \frac{1}{2} \int_{|y| > 1} \widehat \Phi(y) \> dy + O\left( \frac{1}{\log Q}\right).}
    
\end{lem}

\begin{proof}
Explicitly,     
\est{\Sigma_{triv, main}^{\pm} &= \frac{2\pi (k-1)}{12N^{\pm}(Q)\log Q} \sum_{\substack{L, M \\ (L, M) = 1 \\ L \leq Q^{\epsilon}}} \mu^2(L) \mu(M) M \sqrt{LM} \Psi\left( \frac{LM}{Q}\right)  \sum_{\substack{(m, M) = 1 \\ m \leq Q^{\epsilon}}} \frac{1}{m}   \\
&\quad \times  \sum_{c} \frac{\phi(cM)}{c M} \frac{\mu\left( \frac{cM}{(c, m^2)} \right)}{\phi\left(\frac{cM}{(c, m^2)}\right)} \frac{\mu\left( \frac{c}{(c,L)}\right)}{\phi\left( \frac{c}{(c,L)}\right)} 
\int_1^{\infty} \frac{ 1}{\sqrt{t}} \widehat \Phi\left(\frac{\log t}{\log Q} \right)J_{k-1}\left( \frac{4\pi m \sqrt{tL}}{c\sqrt M}\right) \> dt.}
By the change of variable $y = \frac{\log t}{\log Q}$, the integral over $t$ becomes 
\es{ \label{eqn:intt_to_inty}\log Q \int_{0}^{\infty}Q^{y/2} \widehat \Phi(y) J_{k-1}\left( \frac{4\pi m \sqrt{Q^{y}L}}{c\sqrt M}\right)\> dy.}
Moreover,
\est{ \log Q \int_{-\infty}^{0}Q^{y/2} \widehat \Phi(y) J_{k-1}\left( \frac{4\pi m \sqrt{Q^{y}L}}{c\sqrt M}\right)\> dy &\ll \log Q \int_{-\infty}^{0}Q^{y/2} |\widehat \Phi(y)| \left( \frac{m\sqrt{Q^{y}L}}{c \sqrt M}\right)^{k - 1} \> dy \\
& \ll Q^{\epsilon}\left( \frac{m\sqrt L}{c \sqrt M} \right)^{k-1}.}
Thus, the contribution from this integration is bounded by 
\es{\label{eqn:boundforintywhenyless0} &\ll \frac{Q^{\epsilon}}{Q^2} \sum_{L \leq Q^{\epsilon}} L^{k/2} \sum_{m \leq Q^{\epsilon}} m^{k - 2} \sum_{c} \frac{1}{c^{k + 1}} \sum_{M \sim \frac{Q}{L}} \frac{1}{M^{k/2 - 1}} \ll \frac{Q^{2k \epsilon}}{Q^{k/2 }} \ll Q^{-\epsilon}.}

By Equations \eqref{eqn:intt_to_inty}, \eqref{eqn:boundforintywhenyless0} and the Mellin inversion of $J$-Bessel function in Lemma \ref{jbessel} and of $\Psi$, we have, for $0 < \sigma_1 < k - 1$ and $ \sigma_1 + 2\sigma_2 > 3$,

\est{\Sigma_{triv, main}^{\pm} 
&= \frac{\pi(k-1)}{12N^{\pm}(Q)} \frac{1}{(2\pi i)^2} \int_{(\sigma_1)} \int_{(\sigma_2)}  (2\pi)^{w} \frac{\Gamma \left( \frac{k - w - 1}{2}\right)}{\Gamma \left( \frac{k +  w + 1}{2} \right)} \widetilde{\Psi}(s) Q^{s}\\
&\quad\times \sum_{\substack{L \leq Q^{\epsilon}} } \frac{\mu^2(L)}{L^{-1/2 - w/2 + s}} \sum_{\substack{ m \leq Q^{\epsilon}}} \frac{1}{m^{1 - w}} \,\sum_{(M, mL) = 1} \frac{\mu(M)}{M^{-1/2 + w/2 + s}}    \\
&\quad\times \sum_c \frac{\phi(c M)}{c^{1 + w}} \frac{\mu\left( \frac{cM}{(c, m^2)} \right)}{\phi\left(\frac{cM}{(c, m^2)}\right)} \frac{\mu\left( \frac{c}{(c,L)}\right)}{\phi\left( \frac{c}{(c,L)}\right)}   
\int_{-\infty}^{\infty} Q^{y(w/2 + 1/2)} \widehat \Phi\left(y \right)  \> dy \> ds \> dw + O(Q^{-\epsilon}).}

Next we integrate by parts twice with respect to $y$ to ensure absolute convergence in preparation for shifting contours.  Moreover, we apply Lemma \ref{lem:sumMandsumc} to the sums over $M$ and $c$. The result is that 
\est{\Sigma_{triv, main}^{\pm} &= 
\frac{\pi(k-1)}{12N^{\pm}(Q)(\log Q)^2} \frac{1}{(2\pi i)^2} \int_{(\sigma_1)} \int_{(\sigma_2)}  (2\pi)^{w} \frac{\Gamma \left( \frac{k - w - 1}{2}\right)}{\Gamma \left( \frac{k +  w + 1}{2} \right)} \widetilde{\Psi}(s) Q^{s}\\
&\quad\times \sum_{\substack{L \leq Q^{\epsilon}} } \frac{\mu^2(L)}{L^{-1/2 - w/2 + s}} \sum_{\substack{ m \leq Q^{\epsilon}}} \frac{1}{m^{1 - w}} \,\zeta(2 + w) \zeta(-1/2 + w/2 + s) H(w, s; L, m)  \\
&\quad\times  \frac{4}{(w + 1)^2}  
\int_{-\infty}^{\infty} Q^{y(w/2 + 1/2)} \widehat \Phi''\left(y \right)  \> dy \> ds \> dw + O(Q^{-\epsilon}).}
We shift the contour integration in  $w$ and $s$ to $\tRe(w) = - 1 + \frac{1}{\log Q}$ and $\tRe(s) = \frac 32 +\epsilon$, crossing a simple pole at $s= 3/2 - w/2.$  We write
\es{ \label{eqn:SigmaResat3/2-w/2andI1} \Sigma_{triv, main}^{\pm} = \mathcal R_1^{\pm} + \mathcal I_{1}^{\pm} + O(Q^{-\epsilon}),}
where 
\es{ \label{def:R1} \mathcal R_1^{\pm} &= -\frac{\pi(k-1)}{12N^{\pm}(Q) (\log Q)^2} \frac{1}{2\pi i} \int_{\left(\sigma_1\right)}   (2\pi)^{w} \frac{\Gamma \left( \frac{k - w - 1}{2}\right)}{\Gamma \left( \frac{k +  w + 1}{2} \right)} \widetilde{\Psi}(3/2 - w/2) \\
&\qquad\times \sum_{\substack{L \leq Q^{\epsilon}} } \frac{\mu^2(L)}{L^{1 - w}} \sum_{\substack{ m \leq Q^{\epsilon}}} \frac{1}{m^{1 - w}} \,\zeta(2 + w)  H\left(w, \tfrac 32 - \tfrac w2; L, m\right)  \\
&\qquad\times  \frac{4}{(w + 1)^2}  
\int_{-\infty}^{\infty} Q^{(y - 1)w/2 + y/2 + 3/2} \widehat \Phi''\left(y \right)  \> dy \> dw}
and 
\est{\mathcal I_1^{\pm} &= \frac{\pi(k-1)}{12N^{\pm}(Q)(\log Q)^2} \frac{1}{(2\pi i)^2} \int_{\left(-1 + \frac{1}{\log Q}\right)} \int_{\left(\frac 32 + \epsilon\right)}  (2\pi)^{w} \frac{\Gamma \left( \frac{k - w - 1}{2}\right)}{\Gamma \left( \frac{k +  w + 1}{2} \right)} \widetilde{\Psi}(s) Q^{s}\\
&\qquad\times \sum_{\substack{L \leq Q^{\epsilon}} } \frac{\mu^2(L)}{L^{-1/2 - w/2 + s}} \sum_{\substack{ m \leq Q^{\epsilon}}} \frac{1}{m^{1 - w}} \,\zeta(2 + w) \zeta(-1/2 + w/2 + s) H(w, s; L, m)  \\
&\qquad\times  \frac{4}{(w + 1)^2}  
\int_{-\infty}^{\infty} Q^{y(w/2 + 1/2)} \widehat \Phi''\left(y \right)  \> dy \> ds \> dw. }
By RH (and in particular the Lindel\"of bound), $ \displaystyle \zeta\left(-\tfrac 12 + \tfrac w2 + s\right) \ll ( 1  + |\tIm\left(\tfrac w2 + s\right)|)^{\epsilon}$ and $\displaystyle \zeta(2 + w) \ll |\tIm w|^{\epsilon} + \log Q$.  The integrals over $w, s$ are absolutely convergent and we have 
 \es{ \label{eqn:boundI1}\mathcal I_1^{\pm} \ll Q^{-1/2 + \epsilon}.}
Next we write $\mathcal R_1$ as 
$$ \mathcal R_1^{\pm} = \mathcal R_{1, y \leq 1}^{\pm} + \mathcal R_{1, y > 1}^{\pm},$$
where $\mathcal R_{1, y \leq 1}^{\pm}$ is the contribution when $y \leq 1$ in \eqref{def:R1}, and $\mathcal R_{1, y > 1}^{\pm}$ is the contribution from the complementary  range $y > 1.$
\vskip 0.1in
{\bf Case 1: $y \leq 1$.} For this case, we shift the integral in $w$ to $\tRe(w) = -\tfrac 12$.  Here $\tRe(w/2 + 1/2) > 0$, and  
\est{ \int_{-\infty}^{1} Q^{y(w/2 + 1/2) + 3/2 - w/2} \widehat \Phi''(y) \> dy &\ll \frac{Q^2}{ (w + 1)\log Q}.}
Moreover, the sums over $L, m$ and the integral in $w$  are absolutely convergent. Hence, 
\est{\mathcal R_{1, y \leq 1}^{\pm} &\ll \frac{1}{(\log Q)^3}.}
\vskip 0.1in 

{\bf Case 2: $y > 1$.} We shift the contour integration in $w$ to $\tRe(w) = - 1 + \epsilon$. We pick up triple poles from $\frac{\zeta(2 + w)}{(w + 1)^2}$ at $w = - 1$. The remaining integral can be bounded by $\frac{1}{(\log Q)^3}$, using the same argument to Case 1 since $\tRe(w/2 + 1/2) < 0$ and
\est{ \int_{1}^{\infty} Q^{y(w/2 + 1/2) + 3/2 - w/2} \widehat \Phi''(y) \> dy &\ll \frac{Q^2}{ (w + 1) \log Q}.}
We denote the contribution from the residue as $\mathcal R_2^{\pm}$. Explicitly, 
\est{\mathcal R_2^{\pm} &= \frac{4\pi(k-1)}{24N^{\pm}(Q) (\log Q)^2}  \int_{1}^{\infty} Q^{y/2 + 3/2} \frac{d^2}{dw^2} \left[Q^{(y-1)w/2} \mathscr G(w)\right]_{w = -1} \widehat \Phi''\left(y \right)  \> dy, }
where
\est{ \mathscr G(w) = (2\pi)^{w} \frac{\Gamma \left( \frac{k - w - 1}{2}\right)}{\Gamma \left( \frac{k +  w + 1}{2} \right)} \widetilde{\Psi}(3/2 - w/2)  \sum_{\substack{L \leq Q^{\epsilon}} } \frac{\mu^2(L)}{L^{1 - w}} \sum_{\substack{ m \leq Q^{\epsilon}}} \frac{1}{m^{1 - w}} \,  H\left(w, \tfrac 32 - \tfrac w2; L, m\right).  }
We have
\est{&\frac{d^2}{dw^2} \left[Q^{(y-1)w/2} \mathscr G(w)\right]_{w = -1} \\
& \qquad  = Q^{(-y/2 + 1/2)}\left(\frac{\log^2 Q (y - 1)^2 }{4} \mathscr G(-1) + \log Q (y - 1) \mathscr G'(-1) + \mathscr G''(-1)\right). } 
Hence from above and the integration by parts twice, we obtain that
\est{\mathcal R_2^{\pm} &= \frac{\pi(k-1)Q^{2}}{24N^{\pm}(Q)} \mathscr G(-1)  \int_{1}^{\infty}  (y - 1)^2 \widehat \Phi''\left(y \right)  \> dy + O\left( \frac{1}{\log Q}\right)\\
&= \frac{\pi(k-1)Q^{2}}{12N^{\pm}(Q)} \mathscr G(-1)  \int_{1}^{\infty}   \widehat \Phi\left(y \right)  \> dy + O\left( \frac{1}{\log Q}\right).}
Next we compute $\mathscr G(-1).$ 
\est{\mathscr G(-1) &= \frac{1}{2\pi} \widetilde \Psi(2) \sum_{L \leq Q^{\epsilon}} \frac{\mu^2 (L)}{L^2} \sum_{m \leq Q^{\epsilon}} \frac{1}{m^2} H(-1, 2; L, m).}
By Equation \ref{Hp(-1, 2)} in Lemma \ref{lem:sumMandsumc}, we have if $m$ or $L$ is not 1, then $H(-1, 2; L, m) = 0$. Thus 
\est{\mathscr G(-1) &= \frac{1}{2\pi} \widetilde \Psi(2) \prod_{p} \left( 1 - \frac{2}{p^2} + \frac{1}{p^3}\right) = \frac{1}{2\pi} \widetilde \Psi(2)  T(2),} 
where $T(s)$ is defined in \eqref{def:Ts}. Thus by Lemma \ref{lem:asympforNQ},

\est{\mathcal R_2^{\pm} =  \int_{1}^{\infty} \widehat \Phi(y) \> dy + O\left( \frac{1}{\log Q}\right)= \frac{1}{2}\int_{|y| > 1}\widehat \Phi(y) \> dy + O\left( \frac{1}{\log Q}\right).} 
Putting this together, we have 
\est{\Sigma_{triv, main}^{\pm} &= \mathcal R_1^{\pm} + O(Q^{-\epsilon}) = \mathcal R_2^{\pm} + O\left(\frac{1}{(\log Q)^3}\right) = \frac{1}{2}\int_{|y| > 1}\widehat \Phi(y) \> dy + O\left( \frac{1}{\log Q}\right)}
as desired.
\end{proof}

\section{Proof of Lemma \ref{lem:nontrivchar} -- the  contribution from non-trivial characters} 

The proof of Lemma \ref{lem:nontrivchar} requires the following Lemma.
\begin{lem} \label{lem:sumseparatecM}
Let $(j, M) = 1,$ and $\chi$ be a Dirichlet character modulo $c.$ Then 
\est{\sumstar_{b \mod {cM}} \e{\frac{jb}{cM}} \chi(b) = \mu(M) \chi(M) G_\chi(j),}    
where 
\es{\label{def:Gausssum} 
G_{\chi}(n) = \sum_{a \mod c} \chi(a) \e{\frac{an}{c}}.
}

\end{lem}
\begin{proof}
For $(c, M) = 1$, by the Chinese Remainder Theorem, the sum is
\est{ &\sumstar_{x \mod {c}} \sumstar_{y \mod {M}}  \e{\frac{jx\bar M}{c}}\e{\frac{jy\bar c}{M}}  \chi(x M\bar M + y c \bar c ) \\
&\qquad \qquad = \mu(M) \sumstar_{x \mod {c}}\e{\frac{jx\bar M}{c}} \chi(x) = \mu(M) \chi(M) G_\chi(j)  }
since the sum over $y$ is the Ramanujan's sum and the last equation results from a change of variables.

Now, suppose $(c, M) \neq 1$.  By multiplicity, it suffices to examine the case $c = p^{\alpha}$ and $M = p^{\beta}$ for some $\alpha, \beta \geq 1.$  Thus, $(j, p) = 1,$ and 
\est{\sumstar_{b \mod {p^{\alpha + \beta}}} \e{\frac{jb}{cM}} \chi(b) &= \sum_{0 \leq \ell < p^{\beta} } \sumstar_{b' \mod {p^{\alpha}}}  \e{\frac{j(\ell p^{\alpha} + b')}{p^{\alpha + \beta}}} \chi(b') \\
&= \sum_{0 \leq \ell < p^{\beta} } \e{\frac{j\ell }{p^{\beta}}}\sumstar_{b' \mod {p^{\alpha}}}  \e{\frac{j b'}{p^{\alpha + \beta}}} \chi(b') \\
&= 0,}
which suffices since $\chi(M) = 0$ for $(c, M) > 1$.
\end{proof}

The contribution from non-trivial characters is 
\es{ \label{eqn:nontrivcont}\Sigma_{non}^{\pm} &= \frac{2\pi (k-1)}{12N^{\pm}(Q)\log Q} \sum_{\substack{L, M \\ (L, M) = 1 \\ L \leq Q^{\epsilon}}} \mu^2(L) \mu(M)  \sqrt{LM} \Psi\left( \frac{LM}{Q}\right)  \sum_{\substack{(m, M) = 1 \\ m \leq Q^{\epsilon}}} \frac{1}{m}   \\
&\qquad \times  \sum_{c} \frac{1}{c} \sumstar_{b \mod {cM}} \sumstar_{d\mod c} \e{\frac{m^2b}{cM}}  \e{\frac{\bar b dL}{c}} \frac{1}{\phi(c)} \sum_{\chi \neq \chi_0 \mod c}\overline{\chi(d)} \\
& \qquad \times
\sum_{p\nmid L} \frac{ \log p}{\sqrt{p}} \chi(p)\widehat \Phi\left(\frac{\log p}{\log Q} \right)J_{k-1}\left( \frac{4\pi m\sqrt{pL}}{c\sqrt M}\right). }

Since $(m, M) = 1,$ we apply Lemma \ref{lem:sumseparatecM} to derive that  
\est{\sumstar_{b \mod {cM}} \e{\frac{m^2b}{cM}} \sumstar_{d\mod c}   \e{\frac{\bar b dL}{c}} \overline {\chi(d)} &= G_{\overline{\chi}}(L)\sumstar_{b \mod {cM}} \e{\frac{m^2b}{cM}} \overline{\chi(b)} \\
&= G_{\overline{\chi}}(L) G_{\overline{\chi}}(m^2) \mu(M) \overline{\chi(M)}.}
Thus,
\es{ \label{eqn:nontrivcont}\Sigma_{non}^{\pm} &= \frac{2\pi (k-1)}{12N^{\pm}(Q)\log Q} \sum_{\substack{ L \leq Q^{\epsilon}}} \mu^2(L) \sqrt{L}    \sum_{\substack{(m, M) = 1 \\ m \leq Q^{\epsilon}}} \frac{1}{m} \sum_{c\ge 1} \frac{1}{c \phi(c)}\\
& \qquad \times     \sum_{\chi \neq \chi_0 \mod c} G_{\overline{\chi}}(L) G_{\overline{\chi}}(m^2) \sum_{(M, L) = 1}\mu^2(M) \overline{\chi(M)}  \sqrt{M} \Psi\left( \frac{LM}{Q}\right) \\
& \qquad \times
\sum_{p\nmid L} \frac{ \log p}{\sqrt{p}} \chi(p)\widehat \Phi\left(\frac{\log p}{\log Q} \right)J_{k-1}\left( \frac{4\pi m\sqrt{pL}}{c\sqrt M}\right). }

We introduce a smooth partition of unity for the sum over $p$, and write
$$\sum_p \textup{...} = \sumd_P \sum_p V\bfrac{p}{P}\textup{...},
$$where $\sumd_P$ denotes a sum over $P = 2^j$ for $j\ge 0$, and $V$ is a smooth function compactly supported on $[1/2, 3]$ satisfying $\sumd_P V\bfrac{p}{P} = 1$ for all $p \geq 1$. 
We define 
\es{ \label{def:S(p; c, L, m)}\mathscr S(P; c, L, m) &:= \sum_{(M, L) = 1}\mu^2(M) \overline{\chi(M)}  \sqrt{M} \Psi\left( \frac{LM}{Q}\right) \\
& \qquad \times\sum_{p\nmid L} \frac{ \log p}{\sqrt{p}} \chi(p) V\left(\frac pP \right) \widehat \Phi\left(\frac{\log p}{\log Q} \right)J_{k-1}\left( \frac{4\pi m\sqrt{pL}}{c\sqrt M}\right).}
Our smooth function $\Psi$ forces $M \sim \frac{Q}{L}.$  Hence, the transition region of the Bessel function is 
$$ \frac{m L\sqrt {P}}{c\sqrt Q} \sim 1.$$
By the support of $\widehat \Phi$, we have 
\es{\label{boundforP} P \ll Q^{3 - \delta}.}
We separate into two cases according to the size of $c,$
so 
\est{\Sigma_{non}^{\pm} = \Sigma_{non, sm \ c}^{\pm} + \Sigma_{non, big \ c}^{\pm},}
where $\Sigma_{non, sm \ c}^{\pm}$ is the contribution from $ c \ll \frac{mL \sqrt{P}}{\sqrt Q}$, and $\Sigma_{non, big \ c}^{\pm}$ is the contribution from $c \gg \frac{mL \sqrt{P}}{\sqrt Q}.$  It now remains to prove the following lemmas.

\begin{lem} \label{lem:nontrivsmallc} Assume GRH. Let $P \ll Q^{3 - \delta}$. Then
    $$\Sigma_{non, sm \ c}^{\pm} \ll Q^{-\epsilon},$$
    where the implied constant depends on $\delta$ and $\epsilon.$
\end{lem}

\begin{lem} \label{lem:nontrivbigc} Assume GRH. Let $P \ll Q^{3 - \delta}$. Then
    $$\Sigma_{non, big \ c}^{\pm} \ll Q^{-\epsilon}, $$
    where the implied constant depends on $\delta$ and $\epsilon.$
\end{lem}

\subsection{Proof of Lemma \ref{lem:nontrivsmallc} - small $c$ }

From the Mellin integral representation of $J_{k - 1}$ in \eqref{eqn:ILMellinforJBessel} , 
\est{J_{k-1}\left( \frac{4\pi m\sqrt{pL}}{c\sqrt M}\right) = \frac{1}{4\pi i} \int_{(\sigma)} (2\pi)^{w} \frac{\Gamma\left(\frac{k - w - 1}{2}\right)}{\Gamma\left( \frac{k + w + 1}{2} \right)} \left( m\frac{\sqrt {pL}}{c\sqrt M} \right)^w  \> dw }
where $0 < \sigma < k -1. $
Also, by the Mellin inversion of $\Psi$, equation \eqref{def:S(p; c, L, m)} can be written as
\est{\mathscr S(P; c, L, m) &= \frac{1}{2(2\pi i)^2} \int_{(\sigma_1)} \int_{(\sigma_2)} (2\pi)^{w} \left( \frac{m\sqrt L}{c}\right)^w  \left( \frac{Q}{L} \right)^s\frac{\Gamma\left(\frac{k - w - 1}{2}\right)}{\Gamma\left( \frac{k + w + 1}{2} \right)} \widetilde{\Psi}(s) \\
& \qquad\times \sum_{(M, L) = 1}\frac{\mu^2(M) \overline{\chi(M)} }{M^{-1/2 + w/2 + s}}  \sum_{p\nmid L} \frac{ \log p}{p^{1/2 - w/2}} \chi(p) V\left(\frac pP \right) \widehat \Phi\left(\frac{\log p}{\log Q} \right) \> ds \> dw,}
where $\tRe(w + 2s) > 3$. The sum over $M$ is 
$ L(-1/2 + w/2 + s, \overline {\chi}) H(w, s; \overline {\chi})$ where $H(w, s; \overline {\chi})$ is absolutely convergent when $\tRe(w + 2s) > 2.$ By GRH, for $\sigma \geq 0,$
\es{ \label{eqn:boundforLat1/2withGRH} L\left( \frac 12 + \sigma + it, \chi \right) \ll (1 + |tc|)^{\epsilon}. }
Let $\tRe(w) = \epsilon_0$ for some  $\epsilon_0>0$ to be chosen later depending on the value of $\delta$ in \eqref{boundforP}. We shift the integral in $s$ to $\tRe(s) = 1 + \epsilon_0$. Since $\chi$ is not a principal character, there is no pole. Moreover, by Lemma \ref{lem:CLee3.5}, 
\est{ \sum_{p\nmid L} \frac{ \log p}{p^{1/2 - w/2}} \chi(p) V\left(\frac pP \right) \widehat \Phi\left(\frac{\log p}{\log Q} \right) \ll P^{\epsilon_0/2}\big[(\log P)^{1 + \epsilon}\log (c + |\tIm(w)| ) + \log L\big]. }
Therefore, the integrals over $w$ and $s$ are absolutely convergent, and we obtain that
\est{\mathscr S(P; c, L, m) \ll (\log Q)^{1 + \epsilon} (\log c)\left( \frac{m\sqrt L}{c}\right)^{\epsilon_0} \left( \frac{Q}{L}\right)^{1 + \epsilon_0} P^{\epsilon_0/2}. }
Hence,
\est{\Sigma_{non, sm \ c}^{\pm} &= \frac{2\pi (k-1)}{12N^{\pm}(Q)\log Q} \sumd_{P \ll Q^{3 -\delta}}\sum_{\substack{ L \leq Q^{\epsilon}}} \mu^2(L) \sqrt{L}    \sum_{\substack{(m, M) = 1 \\ m \leq Q^{\epsilon}}} \frac{1}{m} \sum_{\substack{(c, M) = 1 \\ c \ll \frac{mL\sqrt P}{\sqrt Q}}} \frac{1}{c \phi(c)}\\
& \times     \sum_{\chi \neq \chi_0 \mod c} G_{\overline{\chi}}(L) G_{\overline{\chi}}(m^2) \mathscr S(P; c, L, m)\\
&\ll \frac{(\log Q)^{\epsilon}}{Q^{1 - \epsilon_0}} \sumd_{P \ll Q^{3 -\delta}}P^{\epsilon_0/2}\sum_{\substack{ L \leq Q^{\epsilon}}} \frac{1}{L^{1/2 + \epsilon_0/2}}   \sum_{\substack{(m, M) = 1 \\ m \leq Q^{1 + \epsilon}}} \frac{1}{m^{1 - \epsilon_0}} \sum_{\substack{c \ll \frac{mL\sqrt P}{\sqrt Q}}} \frac{1}{c^{1 + \epsilon_0} \phi(c)} \\
&\times  \sum_{\chi \neq \chi_0 \mod c} |G_{\overline{\chi}}(L) G_{\overline{\chi}}(m^2)|. }
By orthogonality, we have
$$ \sum_{\chi \mod c} |G_{\chi}(n)|^2 = (\phi(c))^2,$$ 
so by Cauchy's inequality, it follows that
\es{ \label{eqn:boundforsumofGausssum}\sum_{\chi \neq \chi_0 \mod c} |G_{\overline{\chi}}(L) G_{\overline{\chi}}(m^2)| \leq (\phi(c))^2.} 
Therefore,
\est{\Sigma_{non, sm \ c}^{\pm} &\ll \frac{Q^{\epsilon}}{Q^{1 - \epsilon_0}} \sumd_{P \ll Q^{3 -\delta}}P^{\epsilon_0/2}\sum_{\substack{ L \leq Q^{\epsilon}}} \frac{1}{L^{1/2 + \epsilon_0/2}}   \sum_{\substack{(m, M) = 1 \\ m \leq Q^{\epsilon}}} \frac{1}{m^{1 - \epsilon_0}} \sum_{\substack{(c, M) = 1 \\ c \ll \frac{mL\sqrt P}{\sqrt Q}}} \frac{1}{c^{ \epsilon_0}} \\
&\ll \frac{Q^{2\epsilon} Q^{3\epsilon/2 - 3\epsilon \epsilon_0/2} Q^{3/2 - \delta/2} }{Q^{3/2 - 3\epsilon_0/2}} \ll Q^{-\epsilon}}
when $\epsilon_0$ is chosen to be sufficiently small in terms of $\delta$.

\subsection{Proof of Lemma \ref{lem:nontrivbigc} - big $c$}
From \eqref{jbessel:smallx}, we write the Bessel function as 
\est{J_{k-1}\left(\frac{4\pi m \sqrt{pL}}{c\sqrt M}\right) = \sum_{\ell = 0}^{\infty} \frac{(-1)^{\ell}}{\ell ! (\ell + k -1)!} \left(\frac{2\pi m \sqrt{pL}}{c\sqrt M} \right)^{k-1 + \ell}. }
Therefore, equation \eqref{def:S(p; c, L, m)} is 
\est{\mathscr S(P; c, L, m) &= \sum_{\ell = 0}^{\infty} \frac{(-1)^{\ell}}{\ell ! (\ell + k -1)!} \left(\frac{2\pi m \sqrt{PL}}{c} \right)^{k-1 + \ell} \sum_{(M, L) = 1}\frac{\mu^2(M) \overline{\chi(M)}}{M^{k/2 - 1 + \ell/2}} \Psi\left( \frac{LM}{Q}\right) \\
&\qquad \times\sum_{p\nmid L} \frac{ \log p}{\sqrt{p}} \chi(p) \left( \frac{p}{P}\right)^{\frac{k-1 + \ell}{2}}V\left(\frac pP \right) \widehat \Phi\left(\frac{\log p}{\log Q} \right). }
By Lemma \ref{lem:CLee3.5}, the sum over $p$ is bounded by 
$$ (\log P)^{1 + \epsilon} (\log c) + \log L.$$
For the sum over $M$, by Mellin inversion of $\Psi$, we have for $\sigma > 2 - k/2 - \ell/2,$
\est{ \sum_{(M, L) = 1}\frac{\mu^2(M) \overline{\chi(M)}}{M^{k/2 - 1 + \ell/2}} \Psi\left( \frac{LM}{Q}\right) &= \frac{1}{2\pi i} \int_{(\sigma)} \widetilde \Psi(s) \left( \frac{Q}{L}\right)^s\sum_{(M, L) = 1}\frac{\mu^2(M) \overline{\chi(M)}}{M^{k/2 + \ell/2 - 1 + s}} \> ds \\
&= \frac{1}{2\pi i} \int_{(\sigma)} \widetilde \Psi(s) \left( \frac{Q}{L}\right)^s L\left(\frac k2 +\frac \ell 2- 1 + s, \overline {\chi}\right) H_{\ell}(s, \overline{\chi})\> ds }
where $H_{\ell}(s, \overline{\chi})$ is absolutely convergent when $\tRe(s) > 3/2 - k/2 - \ell/2.$ We shift the contour to $\tRe(s) = 3/2 - k/2 - \ell/2 + \epsilon_0.$  From the bound in \eqref{eqn:boundforLat1/2withGRH}, we have
\est{ \sum_{(M, L) = 1}\frac{\mu^2(M) \overline{\chi(M)}}{M^{k/2 - 1 + \ell/2}} \Psi\left( \frac{LM}{Q}\right) \ll \left( \frac{Q}{L}\right)^{3/2 - k/2 - \ell/2 + \epsilon_0}. }
Therefore
\est{\mathscr S(P; c, L, m) \ll (\log Q)^{1 + \epsilon} (\log c) \sum_{\ell = 0}^{\infty} \frac{1}{\ell ! (\ell + k -1)!} \left(\frac{2\pi m \sqrt{PL}}{c} \right)^{k-1 + \ell} \left( \frac{Q}{L} \right)^{3/2 - k/2 - \ell/2 + \epsilon_0}.}
This and \eqref{eqn:boundforsumofGausssum} implies that
\est{\Sigma_{non, big \ c}^{\pm} &= \frac{2\pi (k-1)}{12N^{\pm}(Q)\log Q} \sumd_{P \ll Q^{3 -\delta}}\sum_{\substack{ L \leq Q^{\epsilon}}} \mu^2(L) \sqrt{L}    \sum_{\substack{(m, M) = 1 \\ m \leq Q^{\epsilon}}} \frac{1}{m} \sum_{\substack{c \gg \frac{mL\sqrt P}{\sqrt Q}}} \frac{1}{c \phi(c)}\\
&\qquad \times     \sum_{\chi \neq \chi_0 \mod c} G_{\overline{\chi}}(L) G_{\overline{\chi}}(m^2) \mathscr S(P; c, L, m) \\
&\ll \frac{Q^{\epsilon}}{Q^2} \sum_{\ell = 0}^{\infty} \frac{(2\pi)^{k - 1 + \ell}}{\ell ! (\ell + k -1)!} \sumd_{P \ll Q^{3 -\delta}}\sum_{\substack{ m, L \leq Q^{\epsilon}}} \frac{\sqrt L}{m} \left( m \sqrt{PL} \right)^{k-1 + \ell}  \\
& \qquad\times \left( \frac{Q}{L} \right)^{3/2 - k/2 - \ell/2 + \epsilon_0}\sum_{c \gg \frac{mL\sqrt P}{\sqrt Q}} \frac{\log c}{c^{k-1 + \ell}} \\
& \ll  \frac{Q^{\epsilon}}{Q^{3/2 - \epsilon_0}} \sumd_{P \ll Q^{3 -\delta}} \sqrt P \ll Q^{-\delta/2 + \epsilon_0 + \epsilon} \ll Q^{-\epsilon}}
upon choosing small enough $\epsilon_0.$

\section{Proof of Corollary \ref{non-vanishing}} \label{non-v}
Let $\mathcal{A}$ denote the class of even Schwartz class functions $\Phi$ with $\Phi(x) \ge 0$ for $x\in \mathbb{R}$ and $\mathrm{supp}(\widehat{\Phi}) \subset (-3,3)$. Our proof of  Corollary \ref{non-vanishing} relies on determining the constants
\begin{equation} \label{extremal}
g^\pm := \inf_{ \substack{\Phi \in \mathcal{A} \\ \Phi(0)>0 } } \frac{1}{\Phi(0)} \int_{-\infty}^\infty \Phi(x) \, \left\{ 1 \pm \frac{\sin 2 \pi x}{2\pi x} \right\} \, dx.
\end{equation}
Using the Fourier pair
\[
\Phi(x) = \Big(\frac{\sin 3\pi x}{3 \pi x} \Big)^2 \quad \text{and} \quad \widehat{\Phi}(\xi) = \frac{1}{3} \max\Big(1-\frac{|\xi|}{3},0\Big)
\]
together with a standard approximation argument, it follows that $g^+ \le \frac{11}{18} $ and $g^- \le \frac{1}{18} $. These bounds are quite good, but not optimal.

There are at least two different approaches that can be used to solve the optimization problem of determining the constants $g^\pm$ defined in \eqref{extremal}.  There is a method developed by Iwaniec, Luo, and Sarnak \cite[Appendix A]{ILS} relying on Fredholm theory and an alternate approach via the framework of reproducing kernel Hilbert spaces of entire functions developed in \cite{CCM,CCLM}. Freeman and Miller \cite{FM2}, extending the ideas in \cite{ILS}, solved a similar optimization problem that we can use to determine the values of $g^\pm$. Let $\mathcal{A}^*$ denote the class of even entire functions $f$ with $f(x) \ge 0$ for $x\in \mathbb{R}$ and $\mathrm{supp}(\widehat{f}) \subset (-3,3)$. Then it follows from  \cite[Theorem 1.1 and Corollary 1.2]{FM2} that
\[
\inf_{ \substack{f \in \mathcal{A}^* \\ f(0)>0 } } \frac{1}{f(0)} \int_{-\infty}^\infty f(x) \, \left\{ 1 + \frac{\sin 2 \pi x}{2\pi x} \right\} \, dx = \frac{5}{12} + \frac{\sqrt{2}}{3 \sqrt{2} + 9 \tan(\frac{1}{2 \sqrt{2}})}
\]
and
\[
\inf_{ \substack{f \in \mathcal{A}^* \\ f(0)>0 } } \frac{1}{f(0)} \int_{-\infty}^\infty f(x) \, \left\{ 1 - \frac{\sin 2 \pi x}{2\pi x} \right\} \, dx = -\frac{3}{4}+ \frac{2}{2 + \sqrt{2} \tan(\frac{1}{2 \sqrt{2}})},
\]
and that these infimums are attained for some $f\in\mathcal{A}^*$. In fact, using Freeman and Miller's work, one can construct the optimal functions. Although the functions $f$ that attain these bounds are entire, they may not belong to the Schwartz class as we are assuming in the hypotheses of Theorem \ref{thm:main}. This is not a major concern  and an approximation argument can be used to prove that 
\begin{equation} \label{FMb}
g^+=\frac{5}{12} + \frac{\sqrt{2}}{3 \sqrt{2} + 9 \tan(\frac{1}{2 \sqrt{2}})} \quad \text{and} \quad 
g^-=-\frac{3}{4}+ \frac{2}{2 + \sqrt{2} \tan(\frac{1}{2 \sqrt{2}})},
\end{equation}
as well (although the infimums are not attained by Schwartz class functions $\Phi \in \mathcal{A}$). The interested reader can find the precise details of this approximation argument in the proof of  \cite[Theorem 2]{CCM}.

We are now in a position to prove Corollary \ref{non-vanishing}

\begin{proof}[Proof of Corollary \ref{non-vanishing}]
For $\Phi \in \mathcal{A}$, on one hand  note that
\[
\begin{split}
\mathscr{OL}^{\pm}(Q) &:= \frac{1}{N^\pm(Q)}\sumb_q \Psi\bfrac{q}{Q} \sum_{f \in \mathcal H_{k}^{\pm}(q)} \sum_{\gamma_f} \Phi \left( \frac{\gamma_f }{2 \pi} \log Q\right)
\\
&\ge \frac{1}{N^\pm(Q)}\sumb_q \Psi\bfrac{q}{Q} \sum_{f \in \mathcal H_{k}^{\pm}(q)} \Phi(0) \cdot \mathop{\mathrm{ord}}\limits_{s=\frac12} L(s,f),
\end{split}
\]
where (by positivity) the inequality follows from considering only the contribution of the zeros at the central point, $s=1/2$, and discarding all other terms in the sum over zeros. On the other hand, 
Theorem \ref{thm:main} implies that
\[
\lim_{Q \to \infty} \mathscr{OL}^{+}(Q) = \int_{-\infty}^\infty \Phi(x) \, \left\{ 1 + \frac{\sin 2 \pi x}{2\pi x} \right\} \, dx
\]
and 
\[
\lim_{Q \to \infty} \mathscr{OL}^{-}(Q) = \Phi(0) + \int_{-\infty}^\infty \Phi(x) \, \left\{ 1 - \frac{\sin 2 \pi x}{2\pi x} \right\} \, dx.
\]
Therefore, since these results hold for every $\Phi \in \mathcal{A}$, we deduce that
\begin{equation}\label{g1}
\limsup_{Q \to \infty} \frac{1}{N^+(Q)}\sumb_q \Psi\bfrac{q}{Q} \sum_{f \in \mathcal H_{k}^{+}(q)} \mathop{\mathrm{ord}}\limits_{s=\frac12} L(s,f) \ \le \ g^+ 
\end{equation}
and
\begin{equation}\label{g2}
\limsup_{Q \to \infty} \frac{1}{N^-(Q)}\sumb_q \Psi\bfrac{q}{Q} \sum_{f \in \mathcal H_{k}^{-}(q)} \mathop{\mathrm{ord}}\limits_{s=\frac12} L(s,f) \ \le \ 1 + g^-.
\end{equation}
We now use the values of $g^\pm$ in \eqref{FMb} to give a lower-bound for the proportion of non-vanishing of $L$-functions and their derivatives at central point in our families given in Corollary \ref{non-vanishing}. For any integer $m \ge 0$, let
\[
P_m^\pm(Q)= \frac{1}{N^\pm(Q)}\sumb_q \Psi\bfrac{q}{Q} \Big| \Big\{ f \in \mathcal H_{k}^{\pm}(q) : \mathop{\mathrm{ord}}\limits_{s=\frac12} L(s,f) = m\Big\} \Big|, 
\]
so that 
\[
\sum_{m=0}^\infty P_m^\pm(Q)=1.
\]
Note that quantities we are interested in bounding below are given by
\[
\liminf_{Q\to \infty} P_0^+(Q)= \liminf_{Q\to \infty} \frac{1}{N^{+}(Q)}  \sumb_q  \Psi\bfrac{q}{Q} \, \Big| \left\{ f \in \mathcal H_k^{+}(q) : L(\tfrac{1}{2},f) \ne 0 \right\} \Big|
\]
and, since $L(\frac{1}{2},f)=0$ when $f \in  H_k^{-}(q)$,
\[
\liminf_{Q\to \infty} P_1^-(Q)= \liminf_{Q\to \infty} \frac{1}{N^{-}(Q)}  \sumb_q  \Psi\bfrac{q}{Q} \, \Big| \left\{ f \in \mathcal H_k^{-}(q) : L'(\tfrac{1}{2},f) \ne 0 \right\} \Big|.
\]
Moreover, by \eqref{g1} and \eqref{g2}, we have
\begin{equation} \label{g12}
\limsup_{Q\to \infty} \sum_{m=1}^\infty m\, P_m^+(Q) = \limsup_{Q\to \infty} \frac{1}{N^+(Q)}\sumb_q \Psi\bfrac{q}{Q} \sum_{f \in \mathcal H_{k}^{+}(q)} \mathop{\mathrm{ord}}\limits_{s=\frac12} L(s,f) \le g^+
\end{equation}
and
\begin{equation} \label{g22}
\limsup_{Q\to \infty} \sum_{m=1}^\infty m\, P_m^-(Q) = \limsup_{Q\to \infty} \frac{1}{N^-(Q)}\sumb_q \Psi\bfrac{q}{Q} \sum_{f \in \mathcal H_{k}^{-}(q)} \mathop{\mathrm{ord}}\limits_{s=\frac12} L(s,f) \le 1+g^-.
\end{equation}
Now we take into consideration the sign of the functional equation of $L(s,f)$ for $f \in  H_k^{\pm}(q)$. For $f \in H_k^{+}(q)$, the order of a zero of $L(s,f)$ at $s=\frac{1}{2}$ is always even, so $P_{2m+1}^+(Q)=0$. Thus, by \eqref{g12}, for any $\varepsilon>0$, we have
\[
1 = P_0^+(Q) + \sum_{m=1}^\infty P_{m}^+(Q) \le P_0^+(Q) + \frac{1}{2} \sum_{m=2}^\infty m \, P_{m}^+(Q) \le P_0^+(Q) + \frac{1}{2} \, g^+ + \varepsilon,
\]
when $Q$ is sufficiently large. 
Hence $P_0^+(Q) \ge 1 - \dfrac{1}{2}\,g^+ - \varepsilon$, and therefore
\[
\begin{split}
\liminf_{Q \rightarrow \infty } \frac{1}{N^{+}(Q)}  \sumb_q & \Psi\bfrac{q}{Q} \, \Big| \left\{ f \in \mathcal H_k^{+}(q) : L(\tfrac{1}{2},f) \ne 0 \right\} \Big|  \ge \frac{19}{24} - \frac{1}{6+9 \sqrt{2} \tan(\frac{1}{2\sqrt{2}})}.
\end{split}
\]
This proves the first assertion of the corollary. To prove the second assertion, for $f \in H_k^{-}(q)$, we observe that the order of a zero of $L(s,f)$ at $s=\frac{1}{2}$ is always odd, so $P_{2m}^+(Q)=0$. Therefore, by \eqref{g22}, for any $\varepsilon>0$, we see that
\[
1= \sum_{m=0}^\infty P_{m}^-(Q) \le  P_1^-(Q) + \frac{1}{3} \sum_{m=3}^\infty m \, P_{m}^-(Q) \le \frac{2}{3} \, P_1^-(Q) + \frac{1}{3} \,(1+g^-) + \varepsilon,
\]
when $Q$ is sufficiently large and so $P_1^-(Q) \ge 1 - \dfrac{1}{2}\,g^- - \varepsilon$. Hence 
\[
\liminf_{Q \rightarrow \infty } \frac{1}{N^{-}(Q)}  \sumb_q  \Psi\bfrac{q}{Q} \, \Big| \left\{ f \in \mathcal H_k^{-}(q) : L'(\tfrac{1}{2},f) \ne 0 \right\} \Big|  \ge \frac{11}{8} - \frac{1}{2\!+\!\sqrt{2} \tan(\frac{1}{2\sqrt{2}})}. 
\]
This completes the proof of the corollary. 
\end{proof}

\section*{Acknowledgements}
 V.C. acknowledges support from  NSF grant DMS-2502599. X.L. acknowledges support from Simons Travel Grant 962494 and NSF grant DMS-2302672. M.B.M. acknowledges support from NSF grant DMS-2401461.



\begin{thebibliography}{9999}


\bibitem{BCL} S. Baluyot, V. Chandee and X. Li, {\it Low-lying zeros of a large orthogonal family of automorphic $L$-functions}, available on arXiv: https://arxiv.org/abs/2310.07606.


\bibitem{BBDDM} O. Barrett, P. Burkhardt, J. DeWitt, R. Dorward, and S.J. Miller, {\it One-level density for holomorphic cusp forms of arbitrary level.} Res. Number Theory, 3: Art. 25, 21,2017.



\bibitem{BM}  V. Blomer and D. Mili\'cevi\'c. {\it The second moment of twisted modular L-functions}. Geom. Funct. Anal., 25(2) (2015), 453 - 516.


\bibitem{CCLM} E. Carneiro, V. Chandee, F. Littmann and M. Milinovich, {\it Hilbert spaces and the pair correlation of zeros of the Riemann zeta function }, J. Reine Angrew. Math. (2017) 729, 51-79.

\bibitem{CCM} E. Carneiro, A. Chirre, and M. B. Milinovich, {\it Hilbert spaces and low-lying zeros of L-functions}, Adv. Math. 410
(2022), part B, Paper No. 108748, 48 pp.










\bibitem{Coh et al} P. Cohen, J. Dell, O. E. Gonz\'alez, G. Iyer, S. Khunger, C.-H. Kwan, S. J. Miller, A. Shashkov, A. S. Reina, C. Sprunger, N. Triantafillou, N. Truong, R. V. Peski, S. Willis, and Y. Yang. {\it On the moments of one-level densities in families of holomorphic cusp forms in the level aspect}.  Algebra Number Theory (2024), vol. 18, no. 10, 46 pp. 




\bibitem{DI} J.-M. Deshouillers and H. Iwaniec, {\it Kloosterman sums and Fourier coefficients of cusp forms}, Invent. Math. 70 (1982), 219-288.




\bibitem{FM2} J. Freeman and S. J. Miller, {\it Determining optimal test functions for bounding the average rank in families of L-functions}, SCHOLAR a scientific celebration highlighting open lines of arithmetic research, 97116, Contemp. Math., 655, Centre Rech. Math. Proc., Amer. Math. Soc., Providence, RI, 2015.





\bibitem{GR} I. S. Gradshteyn and I. M. Ryzhik, {\it Table of Integrals, Series, and Products}, Edited by A.Jeffrey and D. Zwillinger. Academic Press, New York, 7th edition, 2007.





\bibitem{Iwaniec} H. Iwaniec, {\it Topics in classical automorphic 
forms}, Graduate Studies in Mathematics, vol. 17 (American Mathematical Society, Providence, RI, 1997).


\bibitem{ILS} H. Iwaniec, W. Luo and P. Sarnak {\it Low lying zeros of families of L-functions.} Inst. Hautes Etudes Sci. Publ. Math. No. 91 (2000), 55-131 (2001).


\bibitem{KaSa} N. Katz and P. Sarnak, {\it Random matrices, Frobenius eigenvalues, and monodromy.} American Mathematical Society Colloquium Publications, 45. American Mathematical Society, Providence, RI, 1999.






















\bibitem{Watt}  G. N. Watson, {\it A treatise on the theory of Bessel functions}, Cambridge University Press, Cambridge 1944.
\end{thebibliography}
\end{document}